\documentclass[12pt,leqno]{amsart}
\usepackage{amsmath,amssymb,amsthm,amsfonts}
\usepackage{esint}
\usepackage{hyperref}
\hypersetup{colorlinks=true}

\numberwithin{equation}{section}

\DeclareMathOperator{\Lip}{Lip}
\DeclareMathOperator{\divg}{div}
\DeclareMathOperator{\supp}{supp}
\DeclareMathOperator{\loc}{loc}
\newcommand{\R}{\mathbb R}
\newcommand{\wt}{\widetilde}
\newcommand{\wh}{\widehat}
\newcommand{\abs}[1]{\left\vert#1\right\vert}
\newcommand{\br}[1]{\left(#1\right)}
\newcommand{\set}[1]{\left\{#1\right\}}
\newcommand{\om}{\Omega}
\newcommand{\pom}{\partial\Omega}
\newcommand{\pd}{\partial}
\newcommand{\dhalf}{D_t^{1/2}}
\newcommand{\HT}{H_t}
\newcommand{\BMO}{\mathrm{BMO}}
\newcommand{\BMOpar}{\mathrm{BMO}_{\mathrm{par}}}
\newcommand{\Dpar}{\mathbb D}
\newcommand{\LL}{\mathcal L}
\newcommand{\Uc}{\mathcal U}
\newcommand{\Cc}{\mathcal C}
\newcommand{\HS}{\dot{\mathcal{HS}}^{1}_{1,1/2}}
\newcommand{\Nt}{\widetilde N}
\newcommand{\Gam}{\Gamma}
\renewcommand{\d}{\,\mathrm d}
\newcommand{\Cnorm}[1]{\|#1\|_{\Cc}}
\newcommand{\Lipc}{\mathrm{Lip}_c}

\def\Yint#1{\mathchoice
	{\YYint\displaystyle\textstyle{#1}}%
	{\YYint\textstyle\scriptstyle{#1}}%
	{\YYint\scriptstyle\scriptscriptstyle{#1}}%
	{\YYint\scriptscriptstyle\scriptscriptstyle{#1}}%
	\!\dint}
\newcommand{\dint}{\int\!\!\!\!\!\int}
\def\YYint#1#2#3{{\setbox0=\hbox{$#1{#2#3}{\iint}$}
		\vcenter{\hbox{$#2#3$}}\kern-.51\wd0}}
\def\longdash{\mkern-1.5mu{-}\mkern-7.5mu{-}}
\def\fiint{\Yint\longdash}

\theoremstyle{plain}
\newtheorem{theorem}[equation]{Theorem}
\newtheorem{lemma}[equation]{Lemma}
\newtheorem{corollary}[equation]{Corollary}
\newtheorem{proposition}[equation]{Proposition}
\newtheorem{definition}[equation]{Definition}

\theoremstyle{remark}

\begin{document}

\title[The Regularity datum and Dirichlet--Regularity duality]{The Regularity datum
on time-varying graph domains and Dirichlet--Regularity duality}

\author{Martin Dindo\v{s}}
\address{School of Mathematics, 
The University of Edinburgh and Maxwell Institute of Mathematical Sciences, Edinburgh, UK}
\email{M.Dindos@ed.ac.uk}

\begin{abstract}
In this paper we address a basic question that has not been considered in the
literature on the parabolic Regularity problem. Let $\om$ be the region above
the graph of a $\Lip(1,\tfrac12)$ function $\phi$ on
$\R^{n-1}\times\R$. The corresponding $\dot L^p_{1,1/2}$  Regularity boundary space on
$\pom$ is defined by pulling the datum back to the flat space where we have the norm
$\|\nabla_y f\|_{L^p}+\|\dhalf f\|_{L^p}$.  As there is more than one graph parametrisation of such a domain, this could petentially lead to different notions of
$\dot L^p_{1,1/2}(\pom)$. This issue arises only on time-varying domains,
which is likely why the issue has not been noticed before.

We show by example that the boundary spaces can indeed disagree. We then
prove that an additional assumption $\dhalf\phi\in\BMOpar$ (called the
Lewis--Murray condition), resolves the issue completely: the spaces defined
through different graphs agree.\medskip

This gives us a well-defined Regularity datum space on every Lewis--Murray graph
domain. We then extend the duality between the parabolic Dirichlet and
Regularity problems, proved on cylinders by the author and E.~Nystr\"om,
 to such domains. Let
$\LL=-\pd_t+\divg(A\nabla\cdot)$ has bounded, measurable, uniformly elliptic
coefficients. We prove that if the adjoint Dirichlet problem
$(D_{\LL^*})_{p'}$ is solvable, then either the Regularity problem
$(R_\LL)_p$ is solvable or $(R_\LL)_q$ is not solvable for any
$q\in(1,\infty)$. 
\end{abstract}

\maketitle

\section{Introduction}\label{S.intro}

Let $\phi:\R^{n-1}\times\R\to\R$ be $\Lip(1,\tfrac12)$ with constant $\ell$,
that is,
\begin{equation}\label{E.lip}
  \abs{\phi(y,s)-\phi(u,\varsigma)}\ \le\ \ell\br{\abs{y-u}
   +\abs{s-\varsigma}^{1/2}} ,
\end{equation}
and let
\begin{equation}\label{E.dom}
  \om:=\set{(y,y_n,t)\in\R^{n-1}\times\R\times\R:\ y_n>\phi(y,t)},\qquad
  \pi(y,t):=\br{y,\phi(y,t),t} .
\end{equation}
For a function $F:\pom\to\R$, the boundary norm for the parabolic Regularity
problem on $\om$ is classically defined by pull-back through the map $\pi$:
\begin{equation}\label{E.datum}
  \|F\|_{p,e_n}:=\big\|F\circ\pi\big\|_{\dot L^p_{1,1/2}(\R^{n-1}\times\R)}
  =\|\nabla_y(F\circ\pi)\|_{L^p(\R^{n-1}\times\R)}
   +\|\dhalf(F\circ\pi)\|_{L^p(\R^{n-1}\times\R)}.
\end{equation}
The parabolic Regularity problem on time-varying graph domains was studied by
Hofmann and Lewis \cite{HL,HL99} and by Nystr\"om \cite{Nys06}. The boundary
space in these papers is defined through a single fixed graph representation
of the domain. Thus the definition refers to a coordinate system in which
$e_n$ is the vertical direction. Let
\begin{equation}\label{E.ell0}
  \ell_0:=\big\|\nabla_y\phi\big\|_{L^\infty(\R^{n-1}\times\R)}\ \le\ \ell
\end{equation}
be the spatial Lipschitz constant of $\phi$. The geometry below depends on
$\ell_0$, which may be much smaller than the constant $\ell$ above. If
$\nu\in\mathbb S^{n-1}$ is a unit vector whose angle with $e_n$ is less than
$\arctan(1/\ell_0)$, with $\arctan(1/0)$ understood to be $\pi/2$,
then no line in the direction $\nu$ meets $\pom$ twice, and hence
$\om$ is again a graph domain in the coordinate system whose vertical direction
is $\nu$, with a graph function $\phi_\nu$ and corresponding projection map
$\pi_\nu$. Here we choose any orthonormal coordinates on $\nu^\perp$. Two
such choices differ by a spatial rotation of $\R^{n-1}$, which preserves
Lebesgue measure, the $L^p$ norm of the spatial gradient and the $L^p$ norm of
the half-time derivative. Thus only the vertical direction $\nu$ matters.
The quantitative bounds for this change of coordinates are given in
Lemma \ref{L.graphgeometry}. We call such a direction $\nu$ \emph{admissible} and put
\begin{equation}\label{E.datumnu}
 \|F\|_{p,\nu}:=
 \|\nabla_y(F\circ\pi_\nu)\|_{L^p(\R^{n-1}\times\R)}
 +\|\dhalf(F\circ\pi_\nu)\|_{L^p(\R^{n-1}\times\R)}.
\end{equation}

\noindent Hence any two admissible directions give two graph representations
of the same boundary and two possible boundary spaces for the Regularity
problem. To our knowledge, the papers cited above do not compare these spaces
or ask whether the definition is independent of the graph direction. The
definition in
\cite[(20)--(22)]{Nys06}, in particular, fixes one projection throughout.\medskip

\noindent This question is different from the obstruction considered in
\cite[\S1]{Din}. On a time-varying domain the function $u(X,t)$ need not be
defined for all $t$ at a fixed spatial point $X$, so the half-time derivative
of a solution in the domain has no immediate meaning. The main theorem of
\cite{Din} shows on cylinders that this interior half derivative need not be
included in the estimate defining the Regularity problem. For a general
time-varying boundary, \cite[Definition 1.2]{Din} assumes the availability of
a boundary space carrying the required fractional half-regularity, but does
not construct such a space.\medskip

\emph{Hence we ask here directly whether the choice of coordinates, and therefore
of the graph defining $\om$, does matter or not.}\medskip

\noindent The transition map $\pi^{-1}\circ\pi_\nu$ of
$\R^{n-1}\times\R$ onto itself preserves the time variable because the
rotation carrying $e_n$ to $\nu$ is spatial. Its spatial part depends on $t$
through $\phi$, which prevents the transition map from commuting with
$\dhalf$. On a cylinder, $\phi$ does not depend on $t$. Neither does the
spatial part of the transition map, and the two norms in \eqref{E.datum} agree
trivially.\medskip

Our answer has two parts. First, under the additional assumption
$\dhalf\phi\in\BMOpar(\R^{n-1}\times\R)$, originating in the work of Lewis
and Murray \cite{LM}, the spaces agree.

\begin{theorem}\label{T.main}
Let $\phi$ be $\Lip(1,\tfrac12)$ with constant $\ell$ and with
$\dhalf\phi\in\BMOpar(\R^{n-1}\times\R)$, let $\om$ be as in \eqref{E.dom} and
let $\nu$ be admissible. Then the function $\phi_\nu$ which defines $\pom$ in
the coordinate system with vertical direction $\nu$ is also
$\Lip(1,\tfrac12)$, with
$\dhalf\phi_\nu\in\BMOpar(\R^{n-1}\times\R)$. Moreover, for every
$1<p<\infty$ and every measurable $F$ on $\pom$,
\begin{equation}\label{E.main}
  \|F\|_{p,\nu}\ \approx\ \|F\|_{p,e_n} ,
\end{equation}
with constants depending only on $n$, $p$, $\ell$, $d_\nu^{-1}$ and
$\|\dhalf\phi\|_{\BMOpar}$, where $d_\nu>0$ is the transversality margin
defined in \eqref{E.transversality}. Hence the space $\dot L^p_{1,1/2}(\pom)$ and its
norm, up to equivalence, do not depend on the direction in which $\om$ is
written as a graph.
\end{theorem}

Theorem \ref{T.counter} shows that the condition
$\dhalf\phi\in\BMOpar$ cannot be removed.\medskip

\noindent We note that this condition is a property of the set $\om$.
By \cite[Theorem 1.2]{BHMN}, on a $\Lip(1,\tfrac12)$ graph it is equivalent
to parabolic uniform rectifiability of $\pom$ and to caloric measure lying in
parabolic $A_\infty$. Theorem \ref{T.main} therefore connects parabolic
uniform rectifiability with the definition of the datum for the parabolic
Regularity problem.\medskip

We prove Theorem \ref{T.main} using the heat extension of the datum, a change
of variables and the half-time derivative estimate of
\cite[Theorem 1.7]{DLPN}. Thus, remarkably, a statement concerning only the
boundary $\pom$ is proved by passing to an extension in one additional
dimension.\medskip

We conclude this part of the paper with a counterexample showing what happens
when the Lewis--Murray condition is removed.

\begin{theorem}\label{T.counter}
Assume $n\ge2$. There exist a function $\psi\in\Lip(1,\tfrac12)$ with
$\dhalf\psi\notin\BMOpar$, a domain
$$\om_\psi:=\set{(x,x_n,t):x_n>\psi(x,t)},$$
an admissible direction $\nu$ for $\om_\psi$, and a function $f$ on
$\pd\om_\psi$ whose pull-back $f\circ\pi_{e_n}$ belongs to
$C^\infty_c(\R^{n-1}\times\R)$, such that $f\in\dot L^{p}_{1,1/2}$ when
computed in the coordinate system with vertical direction $e_n$, that is,
$$\|\nabla_y (f\circ\pi_{e_n})\|_{L^p}+\|\dhalf (f\circ\pi_{e_n})\|_{L^p}<\infty$$
for every $1<p<\infty$, but $f\notin\dot L^{2}_{1,1/2}$ when computed in the
coordinate system with vertical direction $\nu$, that is,
$$\|\nabla_y (f\circ\pi_{\nu})\|_{L^2}+\|\dhalf (f\circ\pi_{\nu})\|_{L^2}=\infty.$$
The failure is confined to the half-derivative: for every admissible $\nu$ and
every $1\le p\le\infty$ one has $\|\nabla_y(f\circ\pi_\nu)\|_{L^{p}}
\approx\|\nabla_y(f\circ\pi_{e_n})\|_{L^{p}}$, with no hypothesis on $\psi$
beyond $\Lip(1,\tfrac12)$.
\end{theorem}

In the second half of the paper we establish a duality principle relating
solvability of the $L^p$ Regularity boundary value problem for
$\LL=-\pd_t+\divg(A\nabla\cdot)$ on a Lewis--Murray graph domain $\om$ and
solvability of the $L^{p'}$ Dirichlet boundary value problem for the adjoint
operator $\LL^*$ on $\om$. There are
two relevant precedents. For constant-coefficient symmetric systems on
time-varying graph domains, \cite[Theorem 4]{Nys06} proves a direct duality
under Dirichlet solvability on both sides of the graph. For arbitrary bounded,
measurable coefficients, \cite{DN} proves the dichotomy used here on Lipschitz
cylinders. Theorem \ref{T.dual} extends the latter result to Lewis--Murray
graph domains, assuming Dirichlet solvability only on $\om$ and Regularity
solvability at one exponent. The implication from Regularity solvability to
adjoint Dirichlet solvability on time-varying domains was proved in
\cite{DD}.\medskip

The cylindrical dichotomy of \cite{DN} was one of the key ingredients in
\cite{DLP}, which resolved the parabolic Regularity problem on Lipschitz
cylinders for operators whose coefficients satisfy a possibly large Carleson
condition, the parabolic analogue of the DKP condition. The author and several
coauthors are working on the corresponding question on time-varying
Lewis--Murray graph domains. We expect Theorem \ref{T.dual} to play the same
role there.\medskip

\begin{theorem}\label{T.dual}
Let $\phi$ be $\Lip(1,\tfrac12)$ with constant $\ell$ and with
$\dhalf\phi\in\BMOpar(\R^{n-1}\times\R)$, and let $\om$ be as in
\eqref{E.dom}. Let
$$\LL:=-\pd_t+\divg\br{A\nabla\cdot}$$
be a parabolic operator on $\om$ whose coefficient matrix $A$ is bounded,
measurable, and uniformly elliptic, and let $1<p<\infty$. Assume that the
Dirichlet problem $(D_{\LL^*})_{p'}$ is solvable. Then either $(R_{\LL})_p$ is
solvable, or $(R_{\LL})_q$ is not solvable for any $1<q<\infty$.
\end{theorem}

Together, Theorem \ref{T.dual} and \cite{DD} give the two implications
\[ (R_{\LL})_p \implies (D_{\LL^*})_{p'} \quad\mbox{and}\quad
   (D_{\LL^*})_{p'} + (R_{\LL})_q \implies (R_{\LL})_p, \]
for parabolic operators on time-varying graph domains of Lewis--Murray type.\medskip

\noindent The Lewis--Murray condition enters Theorem \ref{T.dual} through
Definition \ref{D.bdyspace} and nowhere else. Sections \ref{S.loc} and
\ref{S.extrapolation} use only the $\Lip(1,\tfrac12)$ character of $\phi$.
Without this condition, the right-hand side of \eqref{E.Rp} would refer to a
particular coordinate system by Theorem \ref{T.counter}. This is why the
condition is retained in the statement.\medskip

In the process of proving this result we establish a localisation theorem. It
says that, for a solution that vanishes on a boundary cylinder, the
nontangential maximal function of its gradient on a smaller cylinder is
controlled by an interior average of that gradient. It is the Dirichlet
counterpart of
\cite[Theorem 1.2]{DLP2}, where the same statement is proved for the Neumann
problem on a Lipschitz cylinder, under solvability of both $(N_\LL)_p$ and
$(D_{\LL^*})_{p'}$.

\begin{theorem}[the local estimate]\label{T.locest}
Let $1<p<\infty$ and let $\om$ be a graph domain as in Theorem \ref{T.dual}
on which $(D_{\LL^*})_{p'}$ is solvable. There is a constant $C$, depending
only on $n$, $p$, $\ell$, $\|\dhalf\phi\|_{\BMOpar}$, the ellipticity constants
of $\LL$, and the solvability constant in $(D_{\LL^*})_{p'}$, such that for
every backward parabolic cylinder
$J_r=J_r(P,\tau):=\set{(Y,s):|Y-P|<r,\ \tau-r^2<s\le\tau}$
centred at $(P,\tau)\in\pom$, every $r>0$, and every weak solution $u$ of
$\LL u=0$ in $J_{K_0r}\cap\om$ with zero Dirichlet data on $J_{K_0r}\cap\pom$,
\begin{equation}\label{E.locest}
  \big\|\Nt\br{\abs{\nabla u}\mathbf 1_{J_r}}\big\|_{L^p(\pom)}
  \ \le\ C\,r^{(n+1)/p}\,\fiint_{J_{K_0r}\cap\om}\abs{\nabla u} ,
\end{equation}
for some $K_0=K_0(n,\ell)>1$.
\end{theorem}

\noindent The Neumann theorem of \cite{DLP2} was stated with an $L^2$ average
on the right-hand side. We prefer to state it using the $L^1$ average, which is what we actually need in Section
\ref{S.extrapolation}.\medskip

The hypothesis of Theorem \ref{T.locest} is solvability for the adjoint
operator. The Neumann counterpart has already been used to interpolate
Neumann solvability \cite{DL,DLP2} and to prove solvability of the Neumann
problem for operators whose coefficients satisfy a Carleson condition
\cite{DLPN}. We expect Theorem \ref{T.locest} to be useful outside the present
paper as well.\medskip

To keep the paper shorter, we do not consider domains that are only
locally given by Lewis--Murray graphs (cf.\ \cite{DDH}), nor do we attempt to
give an intrinsic definition of $\dot L^p_{1,1/2}(\pom)$ which does not use
coordinates. Such a definition would be important and useful for extending the parabolic
Regularity problem to boundaries which are not graphs.\medskip

The paper is organised as follows. Section \ref{S.notation} fixes the
notation and defines the boundary value problems. Section \ref{S.quoted}
collects the parabolic estimates used in the proofs. Section \ref{S.main}
proves Theorem \ref{T.main}, and Section \ref{S.counter} gives the
counterexample of Theorem \ref{T.counter}. Section \ref{S.loc} proves the local
estimate, Theorem \ref{T.locest}. In Section \ref{S.extrapolation} we prove
Theorem \ref{T.dual}, using Sobolev extrapolation in one range of exponents and
finite-atomic Hardy--Sobolev interpolation in the other.\medskip

\noindent{\bf Statement about the use of generative AI:} Generative AI (Anthropic model Opus and OpenAI models gpt-5.6 and gpt-6.0)
were used to draft the \LaTeX version of this file as well during the proof reading process. The author is solely responsible for the mathematical ideas, content, and validity of the proofs presented here.

\section{Notation}\label{S.notation}

Points of $\R^n=\R^{n-1}\times\R$ are $z=(y,s)$.
Points of $\Uc=\R^n_+\times\R$ are $(X,t)=(x,x_n,t)$ with $x\in\R^{n-1}$ and
$x_n>0$, and $\pd\Uc$ is identified with $\R^{n-1}\times\R$. For every
$k\ge1$ we use the parabolic norm on $\R^k\times\R$ given by
\[
 \|(Y,s)\|:=\br{\abs Y^2+\abs s}^{1/2}.
\]
Thus the same notation $B_r$ denotes a boundary ball in
$\R^{n-1}\times\R$, of measure comparable to $r^{n+1}$, or an ambient ball
in $\R^n\times\R$, of measure comparable to $r^{n+2}$. The context specifies
which ball is meant. On $\Uc$ the parabolic distance to the boundary is
$x_n$.

For an exponent $1<r<\infty$ we write $r'=r/(r-1)$. The notation
$\mathbf1_E$ denotes the indicator of $E$, and
\[
 \fiint_E f\,\d\mu:=\frac1{\mu(E)}\int_E f\,\d\mu.
\]
We write $A\lesssim B$ when $A\le CB$ and $A\approx B$ when both inequalities
hold. Dependencies of $C$ are stated when they are not clear from the
hypotheses. On an open subset of $\R^n\times\R$, an unadorned $\nabla$ or
$\divg$ always denotes the spatial operators $\nabla_X$ or $\divg_X$.

$\HT$ is the Hilbert transform in the time variable with symbol
$i\,\operatorname{sgn}\tau$, and $\dhalf$ is the operator with symbol
$\abs\tau^{1/2}$. Thus
\begin{equation}\label{E.algebra}
  \pd_t=\dhalf\HT\dhalf,\qquad \dhalf\dhalf=-\HT\pd_t .
\end{equation}
$\Dpar$ is the operator with symbol $\|(\xi,\tau)\|$, and
$\dot L^p_{1,1/2}(\R^n):=\set{f:\ \|\Dpar f\|_{L^p}<\infty}$, modulo constants.
For $1<p<\infty$,
\begin{equation}\label{E.norms}
  \|\Dpar f\|_{L^p}\ \approx\ \|\nabla_yf\|_{L^p}+\big\|\dhalf f\big\|_{L^p} ,
\end{equation}
by \cite[(2.10)]{DN}. Here $\BMOpar$ is bounded mean oscillation over
parabolic balls. We use $\BMO$ without a subscript only for the usual
one-dimensional space on $\R$.

A measure $\d\mu=m(z,\alpha)\d z\d\alpha\ge0$ on
$\R^n\times(0,\infty)$, or on $\Uc$ with $\alpha=x_n$, is \emph{Carleson}
when
\[
 \Cnorm\mu:=\sup_{z_0\in\R^n,\,r>0}\frac1{\abs{B_r}}
 \int_{B_r(z_0)}\int_0^r m(z,\alpha)\,\d\alpha\d z<\infty.
\]

Let $\om$ be either the graph domain \eqref{E.dom} or $\Uc$, and let
$$\delta(X,t):=\inf_{(q,\tau)\in\pom}\big\|(X,t)-(q,\tau)\big\|$$
be the parabolic distance to the boundary, which on $\Uc$ is $x_n$, and let
\begin{equation}\label{E.surfball}
  \Delta_r(P,\tau):=\pom\cap B_r(P,\tau)
\end{equation}
be the surface ball of radius $r$ at a boundary point. The measure on $\pom$ is
defined by slicing in time,
\begin{equation}\label{E.sigma}
  \sigma(E):=\int_\R\mathcal H^{n-1}\br{E\cap\pom_t}\d t
\end{equation}
for a Borel set $E\subset\pom$. Here
$\om_t:=\set{X:(X,t)\in\om}$ and $\pom_t:=\pd\om_t$, and
$\mathcal H^{n-1}$ is the Hausdorff measure on the Lipschitz boundary
$\pom_t$ of the time slice. We define $\sigma$ this way
rather than as a Hausdorff measure on $\pom$ because $\pom$ is only H\"older of
order $\tfrac12$ in time. Hence its Euclidean Hausdorff dimension may exceed
$n$, and the restriction of $\mathcal H^n$ to $\pom$ need not be locally
finite.\medskip

\noindent For $a>0$ and $(q,\tau)\in\pom$ the nontangential cone is
\begin{equation}\label{E.cone}
  \Gam_a(q,\tau):=\set{(X,t)\in\om:\ \big\|(X,t)-(q,\tau)\big\|
   <(1+a)\delta(X,t)} .
\end{equation}
\noindent For a scalar- or vector-valued function $h$ we put
\begin{equation}\label{E.functionals}
  N_a(h)(q,\tau):=\sup_{\Gam_a(q,\tau)}\abs h,\qquad
  S_a(h)(q,\tau):=\Big(\iint_{\Gam_a(q,\tau)}\abs{\nabla_Xh}^2\delta^{-n}\d X\d t
   \Big)^{1/2},
\end{equation}
\begin{equation}\label{E.Ntilde}
  \Nt_a(h)(q,\tau):=\sup_{(X,t)\in\Gam_a(q,\tau)}
   \Big(\fiint_{B_{\delta(X,t)/2}(X,t)}\abs h^2\Big)^{1/2} .
\end{equation}
The absolute values in these formulas are Euclidean norms, and the spatial
gradient is taken componentwise. We also use the area function
\begin{equation}\label{E.areafun}
 \mathfrak A_a(h)(q,\tau):=
 \Big(\iint_{\Gam_a(q,\tau)}\abs{\pd_th}^2
       \delta^{-n+2}\d X\d t\Big)^{1/2}.
\end{equation}
Thus $S_a(\nabla_Xh)^2$ is the conical integral of
$\abs{\nabla^2_Xh}^2\delta^{-n}$. The $L^p$ norms of $N_a$, $\Nt_a$, $S_a$
and $\mathfrak A_a$ do not depend on the aperture up to constants. We fix one
aperture and suppress $a$, and use subscripts such as $a'$ when the aperture
is enlarged. In particular, $\Nt_a(h)\le N_{a'}(h)$ for a slightly larger
aperture $a'$. The cone and these functionals are attached to the set $\om$
and involve no choice of coordinates.

\noindent The remaining functional is the Carleson one,
\begin{equation}\label{E.carlfun}
  \wt{\mathcal C}(w)(q,\tau):=\sup_{r>0}\frac1{\sigma\br{\Delta_r(q,\tau)}}
  \iint_{B_r(q,\tau)\cap\om}
   \Big(\fiint_{B_{\delta(X,t)/2}(X,t)}\abs w^2\Big)^{1/2}
   \frac{\d X\d t}{\delta(X,t)} ,
\end{equation}
where the inner average is the one in \eqref{E.Ntilde}. This functional is
dual to $\Nt$: for
$1<p<\infty$,
\begin{equation}\label{E.carlduality}
  \big\|\Nt(u)\big\|_{L^p(\pom)}\ \approx\
  \sup\Big\{\iint_\om uF\,\d X\d t:\ \big\|\wt{\mathcal C}(\delta F)
   \big\|_{L^{p'}(\pom)}\le1\Big\} ,
\end{equation}
which is the Carleson duality of Hyt\"onen and Ros\'en \cite{HR} in the form of
\cite[Lemma 2.8]{DLP2}. The normalisation there is by the $(n+1)$-dimensional
parabolic Hausdorff measure of $\pom$, which is comparable to $\sigma$.

\subsection{The graph coordinates}\label{ss.graphgeometry}

Let $\phi$ satisfy \eqref{E.lip}, let $\ell_0$ be as in \eqref{E.ell0}, and
let $\nu$ be admissible. Put
\begin{equation}\label{E.transversality}
 \theta_\nu:=\angle(\nu,e_n),\qquad
 d_\nu:=\cos\theta_\nu-\ell_0\sin\theta_\nu>0.
\end{equation}
No Lewis--Murray assumption is needed in the following lemma.

\begin{lemma}\label{L.graphgeometry}
The rotated graph function satisfies
\begin{equation}\label{E.graphbound}
 |\phi_\nu(y,t)-\phi_\nu(z,s)|
 \le \frac{\sqrt{1+\ell_0^2}}{d_\nu}|y-z|
      +\frac{\ell}{d_\nu}|t-s|^{1/2}.
\end{equation}
The boundary transition has the form
$\Lambda:=\pi^{-1}\circ\pi_\nu$, $\Lambda(y,t)=(\lambda_t(y),t)$.
Both $\Lambda$ and $\Lambda^{-1}$ are Lipschitz for the parabolic distance,
with constants depending only on $\ell$ and $d_\nu^{-1}$. For each fixed $t$,
both $\lambda_t$ and $\lambda_t^{-1}$ are Lipschitz, with constants depending
only on $\ell_0$ and $d_\nu^{-1}$, and almost everywhere
\begin{equation}\label{E.graphjac}
 d_\nu\le |\det D\lambda_t^{-1}|\le1+\ell_0,\qquad
 (1+\ell_0)^{-1}\le |\det D\lambda_t|\le d_\nu^{-1}.
\end{equation}
\end{lemma}

\begin{proof}
For $n=1$ the maps are the identity and the assertions are immediate.
Assume $n\ge2$. Write $\nu=(\nu',\nu_n)$ and let
$Q:\R^{n-1}\to\nu^\perp$ be a linear isometry. Write
$Qy=(q'(y),q_n(y))$. For fixed $(y,t)$ consider
$$h_{y,t}(r):=q_n(y)+\nu_n r-\phi(q'(y)+r\nu',t).$$
Since $\nu_n=\cos\theta_\nu$ and $|\nu'|=\sin\theta_\nu$, for $r>r'$ we have
$$h_{y,t}(r)-h_{y,t}(r')\ge d_\nu(r-r').$$
Thus $h_{y,t}$ is strictly increasing and maps $\R$ onto $\R$. Its unique
zero is $\phi_\nu(y,t)$, and $h_{y,t}(r)>0$ describes $\om$ in these
coordinates. Moreover
\[
 |h_{y,t}(r)-h_{z,s}(r)|
 \le \sqrt{1+\ell_0^2}|y-z|+\ell|t-s|^{1/2}.
\]
Comparing the two zeros gives \eqref{E.graphbound}.\medskip

\noindent The transition and its inverse are
\[
 \lambda_t(y)=q'(y)+\nu'\phi_\nu(y,t),\qquad
 \lambda_t^{-1}(x)=Q^{\mathsf T}(x,\phi(x,t)).
\]
The first formula and \eqref{E.graphbound} give the asserted bounds for
$\lambda_t$ and $\Lambda$. The second gives
\[
 |\lambda_t^{-1}(x)-\lambda_s^{-1}(z)|
 \le \sqrt{1+\ell_0^2}|x-z|+\ell|t-s|^{1/2},
\]
and hence the bounds for their inverses. Finally, differentiating the
second formula gives
$$|\det D\lambda_t^{-1}(x)|
   =\nu_n-\nu'\cdot\nabla_x\phi(x,t).$$
The expression on the right lies between $d_\nu$ and $1+\ell_0$.
The inverse Jacobian formula proves \eqref{E.graphjac}.
\end{proof}

\subsection{The boundary spaces}\label{ss.bdyspaces}

In this section $\om$ is the graph domain \eqref{E.dom}, whose graph function
$\phi$ is $\Lip(1,\tfrac12)$ with constant $\ell$ and has
$\dhalf\phi\in\BMOpar$ with $\|\dhalf\phi\|_{\BMOpar}\le\eta$. We call
$(\ell,\eta)$ the \emph{character} of $\om$.

\begin{definition}[\bf the Regularity boundary space]\label{D.bdyspace}
Let $1<p<\infty$. By \eqref{E.main} of Theorem \ref{T.main}, the norms
$\|F\|_{p,\nu}$ of \eqref{E.datum} and \eqref{E.datumnu} are comparable to
one another for all admissible $\nu$. We may
therefore put
\begin{equation}\label{E.bdyspace}
  \dot L^p_{1,1/2}(\pom):=\set{F:\pom\to\R:\ \|F\|_{p,e_n}<\infty},\qquad
  \|F\|_{\dot L^p_{1,1/2}(\pom)}:=\|F\|_{p,e_n} ,
\end{equation}
and both the space and its norm, the latter up to bounded constants, are
attached to the set $\pom$ and not to a coordinate system.
\end{definition}

\noindent For the Dirichlet problem the boundary space is
$L^p(\pom,\d\sigma)$, with $\sigma$ as in \eqref{E.sigma}.\medskip

\subsection{The operator and the parabolic measure}\label{ss.pm}

The operator is
\begin{equation}\label{E.pde}
  \LL u:=-\pd_tu+\divg\br{A\nabla u},\qquad
  \LL^*u:=\pd_tu+\divg\br{A^{\mathsf T}\nabla u},
\end{equation}
with $A:\om\to M_{n\times n}(\R)$ bounded, measurable and uniformly elliptic,
\begin{equation}\label{E.elliptic}
  \lambda\abs\xi^2\le\sum_{i,j}a_{ij}(X,t)\xi_i\xi_j,\qquad
  \abs{A(X,t)}\le\Lambda ,
\end{equation}
for almost every $(X,t)\in\om$ and all $\xi\in\R^n$. No further hypothesis is
placed on $A$ and the matrix need not be symmetric. Write $C_c(\pom)$ for the
continuous functions on $\pom$ of compact support.

\begin{definition}[\bf the parabolic measure]\label{D.pm}
A family $\set{\omega^{(X,t)}}_{(X,t)\in\om}$ of Borel probability measures on
$\pom$, each supported in $\set{(q,\tau)\in\pom:\tau\le t}$, is the \emph{parabolic
measure} of $\LL$ on $\om$ if for every $f\in C_c(\pom)$ the function
\begin{equation}\label{E.pmsol}
  u(X,t)\ :=\ \int_{\pom}f(q,\tau)\d\omega^{(X,t)}(q,\tau)
\end{equation}
is a bounded weak solution of $\LL u=0$ in $\om$ which extends continuously to
$\overline\om$ with the value $f$ on $\pom$, and satisfies $\abs u\to0$
uniformly as $t\to-\infty$.
\end{definition}

\noindent Theorem \ref{T.pmeasure} below says that such a family exists and is
unique, and that the same holds for $\LL^*$ with the direction of time
reversed. We write $\omega^{*,(X,t)}$ for its measure. We call \eqref{E.pmsol}
the \emph{solution with datum $f$}.

\subsection{The Dirichlet and the Regularity problem}\label{ss.bvp}

Let $\Lipc(\pom)$ be the functions on $\pom$ of compact support that are
Lipschitz for the parabolic distance of $\pom$. For $1<p<\infty$ put
\begin{equation}\label{E.dense}
  \mathcal R_p(\pom):=\set{f\in\Lipc(\pom):\
   \|f\|_{\dot L^p_{1,1/2}(\pom)}<\infty} .
\end{equation}
The finiteness condition is necessary: a function in $\Lipc(\pom)$ is only
H\"older of order $\tfrac12$ in time, as dictated by the parabolic distance,
and its half-time derivative need not lie in $L^p$. The function $b$ of
Proposition \ref{P.counter} is an example.
Since $\Lipc(\pom)\subset C_c(\pom)$, the solution \eqref{E.pmsol} is defined
for every datum of \eqref{E.dense}.

\begin{definition}[\bf the Dirichlet problem]\label{D.Dp}
Let $1<p<\infty$. The $L^p$ Dirichlet problem for $\LL$ on $\om$ is
\emph{solvable}, written $(D)_p$, or $(D_\LL)_p$ when the operator is shown, if there is $C>0$ such that for every $f\in C_c(\pom)$ the solution
$u$ with datum $f$ satisfies
\begin{equation}\label{E.Dp}
  \big\|N(u)\big\|_{L^p(\pom,\d\sigma)}\ \le\
  C\,\|f\|_{L^p(\pom,\d\sigma)} .
\end{equation}
\end{definition}

\begin{definition}[\bf the Regularity problem]\label{D.Rp}
Let $1<p<\infty$. The $L^p$ Regularity problem for $\LL$ on $\om$ is
\emph{solvable}, written $(R)_p$, or $(R_\LL)_p$, if there is $C>0$ such that
for every $f\in\mathcal R_p(\pom)$ the solution $u$ with datum $f$ satisfies
\begin{equation}\label{E.Rp}
  \big\|\Nt(\nabla u)\big\|_{L^p(\pom,\d\sigma)}\ \le\
  C\,\|f\|_{\dot L^p_{1,1/2}(\pom)} .
\end{equation}
\end{definition}

\noindent In both definitions $C$ may only depend on $\LL$, on $p$ and on the
character $(\ell,\eta)$ of $\om$. The same definitions are
applied to $\LL^*$, with $\omega^{*,(X,t)}$ in place of $\omega^{(X,t)}$, and we
then write $(D_{\LL^*})_p$ and $(R_{\LL^*})_p$.\medskip

\noindent The class $C_c(\pom)$ is dense in $L^p(\pom,\d\sigma)$, and
$\mathcal R_p(\pom)$ is dense in $\dot L^p_{1,1/2}(\pom)$. Hence in either case
the solution operator extends to the whole boundary space with the estimate
still holding, and every datum has a solution attaining it nontangentially
almost everywhere.

\noindent For the second density, \eqref{E.lip} makes $\pi$ bi-Lipschitz
for the parabolic distances, and \eqref{E.datum} identifies the boundary
norm with the flat norm. Every compactly supported smooth function on the
flat space is parabolically Lipschitz, by the mean value theorem at distances
at most one and boundedness at larger distances. Such functions are dense
in $\dot L^p_{1,1/2}=\dot F^1_{p,2}$; see \cite[\S2.3]{DN} for this
Triebel--Lizorkin identification. The approximation is modulo constants.
Adding a constant to the datum adds the same constant to its
parabolic-measure solution, so neither side of \eqref{E.Rp} changes.

\noindent We record one consequence, because it is the form in which the
above is used in Section \ref{S.extrapolation}. To verify $(R_\LL)_p$ it is
enough to verify \eqref{E.Rp} for the data $g\circ\pi^{-1}$ with
$g\in C^\infty_c(\R^{n-1}\times\R)$.

\noindent Indeed, let $F\in\mathcal R_p(\pom)$ and put $f=F\circ\pi$, a
compactly supported function which is Lipschitz for the parabolic distance and
has $\|f\|_{\dot L^p_{1,1/2}}<\infty$. Let $f_k$ be the anisotropic
mollifications of $f$ at scales tending to zero. Then $f_k\in
C^\infty_c(\R^{n-1}\times\R)$, the convolution commutes with $\Dpar$ and
contracts $L^p$, so $\|f_k\|_{\dot L^p_{1,1/2}}\le\|f\|_{\dot
L^p_{1,1/2}}$, and $f_k\to f$ uniformly because $f$ is uniformly continuous.
By \eqref{E.maxpr} the solutions converge uniformly on $\overline\om$, hence
locally uniformly in $\om$, and the Caccioppoli inequality then gives
$\nabla u_{f_k\circ\pi^{-1}}\to\nabla u_F$ in $L^2_{\loc}(\om)$. Fatou's
lemma applied first to the Whitney averages of \eqref{E.Ntilde} and then on
$\pom$ gives
$$\big\|\Nt(\nabla u_F)\big\|_{L^p(\pom)}\ \le\
  \liminf_k\big\|\Nt(\nabla u_{f_k\circ\pi^{-1}})\big\|_{L^p(\pom)}
  \ \le\ C\|F\|_{\dot L^p_{1,1/2}(\pom)} ,$$
which is \eqref{E.Rp}.

\section{Quoted parabolic results}\label{S.quoted}

\subsection{The Hofmann--Lewis mollification}\label{ss.moll}

Following \cite[(1.6)]{HLmem} we fix $\vartheta\in C^\infty_0(B_1)$, even and
non-negative, with $\int\vartheta=1$, and write
$$P_\lambda g:=\vartheta_\lambda*g,\qquad
  \vartheta_\lambda(y,s):=\lambda^{-(n+1)}\vartheta\br{y/\lambda,s/\lambda^2},$$
the convolution taken on $\R^{n-1}\times\R$. All the properties of $P_\lambda$
that we use are contained in the following lemma of Hofmann and Lewis. In it
$k$ and $\theta$ are non-negative integers, $\varphi$ is a multi-index in
the $n-1$ spatial variables, and
$$l:=k+\abs\varphi+\theta,\qquad
  \pd^{\,l}:=\frac{\pd^{\,l}}{\pd x_n^k\,\pd x^\varphi\,\pd t^\theta} ,$$
the derivatives being taken of $P_{\gamma x_n}a$ as a function of the point
$(X,t)=(x,x_n,t)$ of $\Uc$, so that $\pd/\pd x_n$ differentiates through the
scale as well.

\begin{lemma}[{\cite[Chapter I, Lemma A]{HLmem}}]\label{L.HL}
Let $a:\R^{n-1}\times\R\to\R$ satisfy
$$\abs{a(y,t)-a(u,t)}\le\ell_0\abs{y-u},\qquad
  \big\|\dhalf a\big\|_{\BMOpar}\le\ell_1 ,$$
with $\ell_0,\ell_1<\infty$. If $l\ge1$, then
\begin{equation}\label{E.mollbds}
  \big|\pd^{\,l}P_{\gamma x_n}a\big|\ \le\ C\,x_n^{1-l-\theta}
\end{equation}
at every $(X,t)\in\Uc$, and if in addition $k+\theta\ge1$ or
$\abs\varphi\ge2$, then the measure with density
\begin{equation}\label{E.mollcarl}
  \big|\pd^{\,l}P_{\gamma x_n}a\big|^2\,x_n^{2l+2\theta-3}
\end{equation}
is Carleson on $\Uc$. Both $C$ and the Carleson norm depend only on $n$, $l$,
$\gamma$, $\ell_0$ and $\ell_1$.
\end{lemma}

\noindent The two hypotheses are \cite[(1.3)]{HLmem} and \cite[(1.4)]{HLmem},
and together they imply that $a$ is $\Lip(1,\tfrac12)$, which is
\cite[(1.5)]{HLmem}. Our $\phi$ satisfies them with $\ell_0$ the spatial
Lipschitz constant \eqref{E.ell0} and $\ell_1=\|\dhalf\phi\|_{\BMOpar}$, and so
does $\phi_\nu$ by Corollary \ref{C.BHMN} below.\medskip

\noindent For the chain rule
$\pd_{x_n}P_{\gamma x_n}a
=\gamma(\pd_\lambda P_\lambda a)|_{\lambda=\gamma x_n}$,
we need the bound $|\pd_\lambda P_\lambda a|\le C(n)\ell$ for
$a\in\Lip(1,\tfrac12)$ with constant $\ell$, independent of $\gamma$.
This is immediate:
$\pd_\lambda P_\lambda a=\lambda^{-1}\br{\lambda\pd_\lambda\vartheta_\lambda}*a$,
and the kernel $\lambda\pd_\lambda\vartheta_\lambda$ has integral zero and is
supported in $B_\lambda$. Hence at a point $z$ it may be applied to $a-a(z)$
instead of to $a$, and $\abs{a-a(z)}\le\ell\lambda$ on $B_\lambda(z)$.

\subsection{Lewis--Murray is a property of the set $\om$}\label{ss.BHMN}

\begin{theorem}[{\cite[Theorem 1.2]{BHMN}}]\label{T.BHMN}
Assume $n\ge2$, let $\om$ be a $\Lip(1,\tfrac12)$ graph domain with graph
function $\phi$, and let $\sigma$ be parabolic surface measure on $\pd\om$.
The following are equivalent:
\begin{enumerate}
\item the caloric measures of $\pd_t-\Delta$ form a parabolic
  $A_\infty(\d\sigma)$ family in the scale-local sense of
  \cite[Definition 3.18]{BHMN};
\item $\dhalf\phi\in\BMOpar$;
\item $\pd\om$ is parabolically uniformly rectifiable.
\end{enumerate}
The implications are quantitative, with constants depending only on the
dimension, the $\Lip(1,\tfrac12)$ character and the quantitative constant of
the corresponding condition.
\end{theorem}

\begin{corollary}\label{C.BHMN}
Let $\phi$ be $\Lip(1,\tfrac12)$ with $\dhalf\phi\in\BMOpar$ and let $\nu$ be
admissible. Then $\dhalf\phi_\nu\in\BMOpar$, and
$\|\dhalf\phi_\nu\|_{\BMOpar}$ is bounded in terms of $n$, of $\ell$, of
$d_\nu^{-1}$ and of $\|\dhalf\phi\|_{\BMOpar}$.
\end{corollary}

\begin{proof}
If $n=1$, the only admissible direction is $e_1$ and there is nothing to
prove. Assume $n\ge2$. By Theorem \ref{T.BHMN}, the condition
$\dhalf\phi\in\BMOpar$ says quantitatively that $\pd\om$ is parabolically
uniformly rectifiable. The latter is a property of the set and its parabolic
surface measure and is invariant under spatial rotations. Apply Theorem
\ref{T.BHMN} once more after rotating $\nu$ to the vertical direction. This
gives $\dhalf\phi_\nu\in\BMOpar$, with the asserted quantitative dependence
because \eqref{E.graphbound} controls the $\Lip(1,\tfrac12)$ character of
$\phi_\nu$ in terms of $\ell$ and $d_\nu^{-1}$.
\end{proof}

\subsection{Existence of the parabolic measure}\label{ss.pmproof}

\begin{theorem}\label{T.pmeasure}
Let $\om$ be as in \eqref{E.dom} and let $\LL$ be as in \eqref{E.pde},
uniformly elliptic in the sense of \eqref{E.elliptic}. Then the parabolic
measure of Definition \ref{D.pm} exists, and the family
$\set{\omega^{(X,t)}}$ is unique. Moreover
\begin{equation}\label{E.maxpr}
  \sup_\om\abs u\ \le\ \sup_{\pom}\abs f ,
\end{equation}
and $u\ge0$ on $\om$ whenever $f\ge0$ on $\pom$. The same holds for $\LL^*$
with the direction of time reversed. Neither $\omega^{(X,t)}$ nor
$\omega^{*,(X,t)}$ depends on the direction in which $\om$ is written as a
graph.
\end{theorem}

\begin{proof}
We only sketch the construction, which is the classical one. Exhaust $\om$ by
the bounded domains
$$\om^R:=\set{(X,t)\in\om:\ \abs x<R,\ x_n<\phi(x,t)+R,\ -R^2<t<R^2} ,$$
and write $F_R$ for the part of the parabolic boundary $\pd_p\om^R$ that does
not lie on $\pom$.\medskip

On each $\om^R$ the continuous Dirichlet problem is solvable by the
Perron--Wiener--Brelot--Bauer method, every point of $\pd_p\om^R$ being regular
because the complement of $\om$ satisfies the measure density condition, and
there is a unique Borel measure $\omega_R^{(X,t)}$ on $\pd_p\om^R$ representing
the solution. This is \cite[\S1]{N}, where the coefficients are taken
qualitatively smooth so that the solution is classical. That hypothesis is
removed by taking weak limits of solutions with smooth coefficients, as in
the case $B=0$ of \cite[Chapter I, \S3]{HLmem}, with estimates independent of
the smoothing.
\cite{N} states its results for symmetric $A$, but the construction just quoted
uses only the maximum principle.

Let $f\in C_c(\pom)$ and, for $R$ so large that $\supp f$ meets $\pd_p\om^R$
only in $\pom$, let $u_R$ be the solution on $\om^R$ with datum $f$ on
$\pom\cap\pd_p\om^R$ and $0$ on $F_R$. Boundary H\"older continuity for $\LL$,
with an exponent $\beta\in(0,1)$ depending only on $n$, $\lambda$, $\Lambda$
and $\ell$ \cite[Chapter I, Lemma 3.9]{HLmem}, gives
$$\omega_R^{(X,t)}(F_R)\ \lesssim\ \Big(\frac{\delta(X,t)}R\Big)^\beta .$$
By the maximum principle the $u_R$ therefore form a Cauchy family, uniformly on
the compact subsets of $\om$, and their limit $u$ is a bounded weak solution of
$\LL u=0$, continuous on $\overline\om$ with the value $f$ on $\pom$. The
restrictions $\omega_R^{(X,t)}\big|_{\pom}$ increase with $R$. Their limit
$\omega^{(X,t)}$ represents $u$ by \eqref{E.pmsol}, is supported in
$\set{(q,\tau)\in\pom:\tau\le t}$ because each $\omega^{(X,t)}_R$ is, and is a probability
measure because $\omega_R^{(X,t)}(\pd_p\om^R)=1$ while
$\omega_R^{(X,t)}(F_R)\to0$. The two assertions of \eqref{E.maxpr} are the
maximum principle on each $\om^R$.\medskip

For uniqueness let $w$ be a bounded weak solution in $\om$, continuous on
$\overline\om$, vanishing on $\pom$, with $\abs w\to0$ uniformly as
$t\to-\infty$. The same decay, applied on $\om^R\cap\set{t>\tau_0}$ and then
letting $R\to\infty$, gives
$$\abs{w(X,t)}\ \lesssim\ \Big(\sup_{\om\cap\set{t=\tau_0}}\abs w\Big)
  \Big(\frac{\delta(X,t)}{(t-\tau_0)^{1/2}}\Big)^\beta ,$$
and $\tau_0\to-\infty$ gives $w=0$. The assertions for $\LL^*$ are the same
with the direction of time reversed, $\LL^*$ being uniformly elliptic with the
same constants.

Finally, every ingredient above: the set $\om$, the operator $\LL$, the
boundary $\pom$ and the parabolic distance is attached to $\om$ and not to a
coordinate system, so neither measure depends on the direction in which $\om$ is
written as a graph.
\end{proof}

\subsection{Boundary and Green-function decay}\label{ss.potdecay}

We use the following local consequences of the parabolic boundary
theory. They are stated on a graph domain, so their constants have no
dependence on a diameter.

If $Z=(X,t)\in\om$, we abbreviate $\omega^{(X,t)}$ to $\omega^Z$, and use
the same convention for the adjoint parabolic measure.

\begin{proposition}\label{P.potdecay}
Let $\om$ be as in \eqref{E.dom}, and let $\LL$ be as in \eqref{E.pde}. There
are $\beta\in(0,1)$ and $C$, depending only on $n$, $\lambda$, $\Lambda$ and
$\ell$, with the following properties.
\begin{enumerate}
\item If $v$ solves $\LL v=0$ in $B_{2R}(P,\tau)\cap\om$ and vanishes on
  $B_{2R}(P,\tau)\cap\pom$, then
  \begin{equation}\label{E.bdryHolder}
   |v(X,t)|\ \le\ C\left(\frac{\delta(X,t)}R\right)^\beta
   \sup_{B_{2R}(P,\tau)\cap\om}|v|
  \end{equation}
  for $(X,t)\in B_R(P,\tau)\cap\om$.
\item The Green function $G(X,t;Y,s)$ exists, vanishes when $t\le s$, and
  satisfies the Gaussian upper bound
  \begin{equation}\label{E.GreenGaussian}
   0\le G(X,t;Y,s)\ \le\
   \frac{C}{(t-s)^{n/2}}
   \exp\left(-\frac{|X-Y|^2}{C(t-s)}\right),\qquad t>s.
  \end{equation}
  If $\rho=\|(X,t)-(Y,s)\|\ge C\delta(Y,s)$, then
  \begin{equation}\label{E.Greenbdry}
   G(X,t;Y,s)\ \le\ C\,\delta(Y,s)^\beta\rho^{-n-\beta}.
  \end{equation}
\item Let $\Delta_r=\Delta_r(P,\tau)$ and $R\ge Cr$.  Suppose that
  $Z_R=(X_R,t_R)$ is time-compatible at scale $R$, in the sense that
  \[
   \delta(Z_R)\approx R,\qquad
   \inf_{z\in\Delta_r}\|Z_R-z\|\approx R,\qquad t_R-\tau\ge cR^2.
  \]
  Then
  \begin{equation}\label{E.omegadecay}
   \omega^{Z_R}(\Delta_r)\ \le\ C\left(\frac rR\right)^{n+\beta}.
  \end{equation}
\item Let $J_R=J_R(P,\tau)$ be a backward parabolic cylinder centred on
  $\pom$. If $v\ge0$ is a solution which vanishes on $J_{CR}\cap\pom$, then
  the boundary Carleson estimate, followed by an interior
  Harnack chain, gives
  \begin{equation}\label{E.boundaryCarleson}
   \sup_{J_R\cap\om}v\ \le\ C v(Z_R^+),
  \end{equation}
  where $Z_R^+$ is a scale-$R$ corkscrew lying a fixed multiple of $R^2$
  later than the cylinder. Fixed changes in the cylinders and in the
  location of $Z_R^+$ only change $C$.
\end{enumerate}
The corresponding statements hold for the adjoint operator after reversing
the direction of time.
\end{proposition}

\begin{proof}
The boundary estimate \eqref{E.bdryHolder} is the boundary oscillation estimate
of De Giorgi--Nash--Moser. In the present geometry it is
\cite[Chapter I, Lemma 3.9]{HLmem}; see also \cite[\S2]{N}. The Green function
is obtained by approximation with smooth coefficients as in
\cite[Chapter I, \S3]{HLmem}, and \eqref{E.GreenGaussian} is the estimate of
\cite{Ar}.

For fixed $(X,t)$ the function $G(X,t;\cdot)$ solves the adjoint equation away
from its pole and has zero boundary values.  We apply \eqref{E.bdryHolder} to
it on the ball of radius $\rho/4$ about the boundary projection of $(Y,s)$.
That ball excludes the pole, because $\delta(Y,s)\le\rho/C$.  On it
$|X-Y'|\approx\rho$ whenever $t-s'$ is small, so \eqref{E.GreenGaussian} and
$\sup_{\varsigma>0}\varsigma^{-n/2}e^{-\rho^2/C\varsigma}\approx\rho^{-n}$
bound $G(X,t;\cdot)$ there by $C\rho^{-n}$.  This gives \eqref{E.Greenbdry}.

The comparison between the Green function and parabolic measure gives
\[
 \omega^{Z_R}(\Delta_r)\ \lesssim\ r^nG(Z_R;A_r),
\]
where $A_r$ is a time-compatible corkscrew point for $\Delta_r$.  This is the
local form of \cite[Chapter I, Lemma 3.10]{HLmem} and
\cite[\S\S3--4]{N}.  Since $\|Z_R-A_r\|\approx R$,
\eqref{E.Greenbdry} and $\delta(A_r)\approx r$ give
\eqref{E.omegadecay}.

The comparison is first applied to continuous boundary cutoffs decreasing to
$\mathbf1_{\Delta_r}$, and the regularity of parabolic measure then gives the
estimate for the surface ball.  These arguments use the equations for $A$ and
$A^{\mathsf T}$ separately and do not require symmetry of $A$.

The estimate \eqref{E.boundaryCarleson} is the classical Carleson estimate
\cite[Lemma 2.5]{N}, together with the interior Harnack inequality and the
Harnack-chain construction of \cite[\S2]{N}.
\end{proof}

\subsection{The Regularity problem for the heat operator on $\Uc$}
\label{ss.flat}

\begin{proposition}\label{P.flat}
Let $1<p<\infty$, let $f\in\dot L^p_{1,1/2}(\R^{n-1}\times\R)$ and let $u$ be
the caloric extension of $f$ to $\Uc$, that is, the solution of $\pd_tu=\Delta u$
in $\Uc$ with $u=f$ on $\pd\Uc$. Then
\begin{equation}\label{E.flat}
  \big\|N(\nabla_Xu)\big\|_{L^p(\R^{n-1}\times\R)}\ \lesssim\
  \|f\|_{\dot L^p_{1,1/2}(\R^{n-1}\times\R)} .
\end{equation}
\end{proposition}

\begin{proof}
This is the solvability of the $L^p$ Regularity problem for the heat
equation on a half-space, the block-form case $A_\parallel=I$ of
\cite[Theorem 1.8]{DLP}.
\end{proof}

\subsection{The square function and the half derivative}\label{ss.NEU}

The two results of this subsection are quoted from \cite{DLPN}. The square
function estimate is stated there for operators without a first-order term;
the half-derivative estimate already allows the drift that occurs here. The
square-function estimate needed for $w$ is obtained directly from its caloric
preimage in Proposition \ref{P.Sw}.

\begin{theorem}[{\cite[Theorem 7.1]{DLPN}}]\label{T.SN}
Let $L=-\pd_t+\divg(A\nabla\cdot)$ on $\Uc$ be uniformly elliptic with
constants $\lambda$ and $\Lambda$, let
\begin{equation}\label{E.carlA}
  \d\mu_A:=\sup_{B_{x_n/2}(X,t)}\br{x_n\abs{\nabla_XA}^2
   +x_n^3\abs{\pd_tA}^2}\d X\d t
\end{equation}
be Carleson, let
\begin{equation}\label{E.boundA}
  x_n\abs{\nabla_XA}+x_n^2\abs{\pd_tA}\ \le\ K<\infty ,
\end{equation}
and let $h$ be an energy solution of $Lh=0$. Then for
every $p\in(0,\infty)$
$$\big\|S(\nabla_Xh)\big\|_{L^p}\ \le\ C\big\|N(\nabla_Xh)\big\|_{L^p},$$
with $C=C(n,p,\lambda,\Lambda,K,\Cnorm{\mu_A})$, provided the left-hand side is
finite.
\end{theorem}

\noindent The supremum over the Whitney ball in \eqref{E.carlA} already gives
\eqref{E.boundA} with $K\lesssim\Cnorm{\mu_A}^{1/2}$. We state the two
separately, as \cite{DLPN} does, because $K$ is what the pointwise estimates
use.

\noindent The proviso $\|S(\nabla_Xh)\|_{L^p}<\infty$ is what the good-$\lambda$
inequality behind the theorem needs and is not part of its conclusion. It is
verified for our solution in Step 2 of the proof of Theorem \ref{T.trace}.

\begin{theorem}[{\cite[Theorem 1.7]{DLPN}}]\label{T.half}
Let
$$L_{\mathbf b}=-\pd_t+\divg_X(A\nabla_X\cdot)
  +\mathbf b\cdot\nabla_X$$
on $\Uc$, where $A$ is as in Theorem \ref{T.SN}. Suppose that $\mathbf b$
has weak first spatial and time derivatives and satisfies
\begin{equation}\label{E.ghyp}
 \begin{aligned}
  x_n\abs{\mathbf b}+x_n^2\abs{\nabla_X\mathbf b}
  +x_n^3\abs{\pd_t\mathbf b}&\le K_{\mathbf b}<\infty,\\
  \d\mu_{\mathbf b}:=
  \br{\abs{\mathbf b}^2x_n+\abs{\nabla_X\mathbf b}^2x_n^3
  +\abs{\pd_t\mathbf b}^2x_n^5}\d X\d t
  &\quad\text{is Carleson}.
 \end{aligned}
\end{equation}
For every $1<p<\infty$, each energy solution of $L_{\mathbf b}h=0$ with
$\Nt(\nabla_Xh),S(\nabla_Xh)\in L^p(\pd\Uc)$ has a boundary trace,
denoted by $\operatorname{Tr}h$, which satisfies
\begin{equation}\label{E.target}
  \big\|\dhalf\operatorname{Tr}h\big\|_{L^p(\pd\Uc)}\ \le\
  C\Big(\big\|\Nt(\nabla_Xh)\big\|_{L^p(\pd\Uc)}
   +\big\|S(\nabla_Xh)\big\|_{L^p(\pd\Uc)}\Big),
\end{equation}
where $C$ depends only on $n$, $p$, $\lambda$, $\Lambda$, $K$,
$\Cnorm{\mu_A}$, $K_{\mathbf b}$ and $\Cnorm{\mu_{\mathbf b}}$.
\end{theorem}

\noindent This is an estimate for a single solution; it requires neither
boundary solvability nor small Carleson norms. This matters here because the
Carleson norms produced by Proposition \ref{P.coeff} below depend on $\ell$,
on $d_\nu^{-1}$ and on $\|\dhalf\phi\|_{\BMOpar}$ and need not be small.

\subsection{A square function for time differences}\label{ss.timesquare}

For a measurable function $f$ on $\R^{n-1}\times\R$, define
\begin{equation}\label{E.timesquaredef}
 \mathcal S_t f(y,t):=
 \left(\int_\R\frac{|f(y,t)-f(y,\tau)|^2}{|t-\tau|^2}\,\d\tau\right)^{1/2}.
\end{equation}
We allow this integral to be infinite.

\begin{lemma}\label{L.timesquare}
For every $1<v<\infty$ and $f\in C_c^\infty(\R^{n-1}\times\R)$,
\begin{equation}\label{E.timesquare}
 \|\mathcal S_t f\|_{L^v}\approx_v\|\dhalf f\|_{L^v}.
\end{equation}
\end{lemma}

\begin{proof}
Stein's characterisation by square differences, in dimension one and
of order $1/2$, gives
\[
 \|g\|_{L^v(\R)}+\|\mathcal S_t g\|_{L^v(\R)}
 \approx_v \|g\|_{L^v(\R)}+\|\dhalf g\|_{L^v(\R)},
 \qquad g\in C_c^\infty(\R).
\]
See \cite[Theorem B, (1.7)]{BJM}. Here we suppress the spatial variable
in \eqref{E.timesquaredef}. The dimension-one exponent range includes
every $v>1$. Apply both inequalities to $g(\lambda\,\cdot)$ and divide
by $\lambda^{1/2-1/v}$. The zeroth-order term becomes
$\lambda^{-1/2}\|g\|_{L^v}$ and tends to zero as $\lambda\to\infty$.
This proves the homogeneous equivalence on $\R$. Apply it at each fixed
$y$ and integrate its $v$-th power in $y$ to obtain \eqref{E.timesquare}.
\end{proof}

\subsection{The atomic endpoint}\label{ss.atoms}

\begin{definition}[{\cite[Definition 6.1]{DN}}]\label{D.atom}
Let $1<\kappa\le\infty$. A function $b$ on $\R^{n-1}\times\R$ is a homogeneous
$(1,\tfrac12,\kappa)$-atom associated with a parabolic ball $B$ of radius $r$ if
\[
  \supp b\subset B,\qquad \|b\|_{L^1}\le r,
\]
and, when $\kappa<\infty$,
\begin{equation}\label{E.atom}
  \|\nabla_yb\|_{L^\kappa}+\big\|\dhalf b\big\|_{L^\kappa}\le\abs B^{-1/\kappa'} .
\end{equation}
When $\kappa=\infty$, \eqref{E.atom} is replaced by
\[
 \|\nabla_yb\|_{L^\infty}+\|\dhalf b\|_{\BMOpar}\le\abs B^{-1}.
\]
The $L^1$ condition follows from the support and gradient conditions by
Poincar\'e's inequality, but we retain it because it is part of the cited
definition. The space
$\dot{\mathcal{HS}}^{1,(\kappa)}_{1,1/2}(\R^{n-1}\times\R)$ consists of the
locally integrable functions $f$ for which
$f=\sum_i\lambda_ib_i$ in the homogeneous Sobolev space with seminorm
$\|\nabla_y f\|_{L^1}+\|\dhalf f\|_{L^1}$, as in \cite[\S2.2]{DN},
with $b_i$ atoms and $\sum_i\abs{\lambda_i}<\infty$, normed by the infimum of
$\sum_i\abs{\lambda_i}$.
\end{definition}

\begin{theorem}[{\cite[Propositions 6.45 and 6.46]{DN}}]\label{T.interp}
The space of Definition \ref{D.atom} does not depend on $\kappa\in(1,\infty]$ and
is written $\HS$. For $1<p<q<\infty$ and $\theta\in(0,1)$ with
$\frac1p=(1-\theta)+\frac\theta q$,
\begin{equation}\label{E.interp}
  \Big(\HS,\ \dot L^q_{1,1/2}\Big)_{\theta,p}\ =\ \dot L^p_{1,1/2} ,
\end{equation}
with equivalent norms.
\end{theorem}

\begin{theorem}[finite-atom interpolation]\label{T.finiteinterp}
Let $1<q\le\kappa<\infty$, and let $\mathfrak T$ be a linearizable sublinear
operator on $C_c^\infty(\R^{n-1}\times\R)+\R$, invariant under the addition
of constants. Here linearizable means that there are a Banach space
$\mathcal B$ and a linear operator $\mathfrak U$ such that
$\mathfrak Tf=\|\mathfrak Uf\|_{\mathcal B}$ pointwise. Assume that
\[
 \|\mathfrak Tf\|_{L^q}\le C_q\|f\|_{\dot L^q_{1,1/2}}
\]
and that
\[
 \|\mathfrak Ta\|_{L^1}\le C_1
\]
for every smooth $(1,\tfrac12,\kappa)$-atom $a$. Then, for every $1<s<q$,
\begin{equation}\label{E.finiteinterp}
 \|\mathfrak Tf\|_{L^s}\le C_s\|f\|_{\dot L^s_{1,1/2}},
 \qquad f\in C_c^\infty(\R^{n-1}\times\R).
\end{equation}
The constant depends only on $n,s,q,\kappa,C_1$ and $C_q$.
\end{theorem}

\begin{proof}
We apply Theorem \ref{T.interp} after mollifying the datum. Linearizability
gives $|\mathfrak Tf-\mathfrak Tg|\le\mathfrak T(f-g)$, so the assumed
$L^q$ estimate extends $\mathfrak T$ continuously to
$\dot L^q_{1,1/2}$ by density, preserving sublinearity.
Choose a nonnegative $\zeta\in C_c^\infty(B_1(0))$ of integral one and put
\[
 \zeta_\epsilon(y,t):=\epsilon^{-(n+1)}
       \zeta(y/\epsilon,t/\epsilon^2),\qquad
 S_\epsilon f:=\zeta_\epsilon*f,\qquad \epsilon>0.
\]
\medskip

\noindent\emph{Mollified atoms.} Every $(1,\tfrac12,\kappa)$-atom $a$ satisfies
\begin{equation}\label{E.atomderivativeLone}
 \|\nabla_y a\|_{L^1}+\|\dhalf a\|_{L^1}\le C_\kappa.
\end{equation}
For completeness, let its supporting ball be $B_r(y_0,t_0)$. The spatial
estimate follows by H\"older's inequality. The same inequality controls
$\dhalf a$ on $\{|y-y_0|<r,\ |t-t_0|<4r^2\}$, whose measure is comparable
to $|B_r|$. Outside this set, $\dhalf a$ vanishes when $|y-y_0|\ge r$,
while the half-derivative kernel gives
\[
 \int_{|t-t_0|\ge4r^2}\int_{\R^{n-1}}|\dhalf a(y,t)|\,\d y\d t
 \lesssim r^{-1}\|a\|_{L^1}\le1.
\]
This proves \eqref{E.atomderivativeLone}, including the nonlocal tail.

The function $S_\epsilon a$ is smooth, supported in $B_{r+\epsilon}(y_0,t_0)$,
and has $L^1$ norm at most $r$. Convolution commutes with both derivatives.
If $\epsilon\le r$, their $L^\kappa$ norms are bounded by contraction and
\eqref{E.atom}. If $\epsilon>r$, Young's inequality and
\eqref{E.atomderivativeLone} give
\[
 \|\nabla_y S_\epsilon a\|_{L^\kappa}
 +\|\dhalf S_\epsilon a\|_{L^\kappa}
 \lesssim\|\zeta_\epsilon\|_{L^\kappa}
 \lesssim\epsilon^{-(n+1)/\kappa'}.
\]
In both cases these norms are bounded by
$C|B_{r+\epsilon}|^{-1/\kappa'}$. Thus $S_\epsilon a$ is a fixed multiple
of a smooth $(1,\tfrac12,\kappa)$-atom, uniformly in $a$, $r$ and
$\epsilon$. In particular,
\begin{equation}\label{E.mollifiedatomendpoint}
 \|\mathfrak T(S_\epsilon a)\|_{L^1}\le C C_1.
\end{equation}
\medskip

\noindent\emph{Atomic series at a fixed scale.} By \eqref{E.atomderivativeLone},
Definition \ref{D.atom}, and Young's inequality,
\begin{equation}\label{E.HSsmoothing}
 \|S_\epsilon h\|_{\dot L^q_{1,1/2}}
 \lesssim\epsilon^{-(n+1)/q'}\|h\|_{\HS},\qquad h\in\HS.
\end{equation}
Here we use the fixed finite atom exponent $\kappa$, which gives an
equivalent $\HS$ norm by Theorem \ref{T.interp}. If
$h=\sum_i\lambda_i a_i$ is such an atomic representation, then, for each
fixed $\epsilon>0$,
\[
 S_\epsilon\Big(\sum_{i=1}^N\lambda_i a_i\Big)
 \longrightarrow S_\epsilon h\quad\hbox{in }\dot L^q_{1,1/2}.
\]
Continuity of $\mathfrak T$ in this norm, an almost everywhere convergent
subsequence in $L^q$, sublinearity, and \eqref{E.mollifiedatomendpoint}
therefore give by Fatou's lemma
\begin{equation}\label{E.regularizedendpoint}
 \|\mathfrak T(S_\epsilon h)\|_{L^1}
 \le C C_1\sum_i|\lambda_i|.
\end{equation}
Taking the infimum over the representations bounds the right side by
$C\|h\|_{\HS}$, independently of $\epsilon$.
Also, convolution and \eqref{E.norms} give
\[
 \|\mathfrak T(S_\epsilon g)\|_{L^q}
 \lesssim C_q\|g\|_{\dot L^q_{1,1/2}}.
\]
\medskip

\noindent\emph{Interpolation and the limit.} For a compatible pair of Banach spaces
$(X_0,X_1)$, write
\[
 K(t,f;X_0,X_1):=\inf_{f=f_0+f_1}
       \big(\|f_0\|_{X_0}+t\|f_1\|_{X_1}\big),\qquad t>0.
\]
The regularized operator $\mathfrak T S_\epsilon$ is defined on
$\HS+\dot L^q_{1,1/2}$ by \eqref{E.HSsmoothing}. For every decomposition
$f=h+g$ in this sum, sublinearity gives
$\mathfrak T(S_\epsilon f)\le\mathfrak T(S_\epsilon h)+
\mathfrak T(S_\epsilon g)$. Splitting the nonnegative function on the left
under these two majorants yields
\[
 K(t,\mathfrak T(S_\epsilon f);L^1,L^q)
 \lesssim K(t,f;\HS,\dot L^q_{1,1/2}),
\]
uniformly in $\epsilon$. Let $\theta\in(0,1)$ satisfy
$1/s=1-\theta+\theta/q$. Taking the real interpolation norm and using
$(L^1,L^q)_{\theta,s}=L^s$ and Theorem \ref{T.interp}, we obtain
\[
 \|\mathfrak T(S_\epsilon f)\|_{L^s}
 \lesssim\|f\|_{\dot L^s_{1,1/2}},
\]
with a constant independent of $\epsilon$. Finally, for smooth compactly
supported $f$, $S_\epsilon f\to f$ in $\dot L^q_{1,1/2}$. The $L^q$
continuity of $\mathfrak T$ and Fatou's lemma along an almost everywhere
convergent subsequence prove \eqref{E.finiteinterp}.
\end{proof}

\begin{lemma}[Poincar\'e--Morrey bound for finite atoms]\label{L.atomsup}
Let $\kappa>n+1$. If $a$ is a $(1,\tfrac12,\kappa)$-atom supported in a
parabolic ball $B$ of radius $r$, then
\begin{equation}\label{E.atomsup}
 \|a\|_{L^\infty}\ \lesssim\ r|B|^{-1}\ \approx\ r^{-n}.
\end{equation}
\end{lemma}

\begin{proof}
The parabolic Morrey inequality and \eqref{E.norms} give
\[
 |a(z)-a(w)|\lesssim
 \|z-w\|^{1-(n+1)/\kappa}
 \big(\|\nabla_ya\|_{L^\kappa}+\|\dhalf a\|_{L^\kappa}\big).
\]
Indeed, this follows by writing $a=\Dpar^{-1}\Dpar a$, splitting the
parabolic Riesz kernel at the distance $\|z-w\|$, and applying H\"older's
inequality to the near and far parts. Choose $w\notin B$ with
$\|z-w\|\lesssim r$. Since $a(w)=0$, the atom normalization gives
\[
 |a(z)|\lesssim r^{1-(n+1)/\kappa}|B|^{-1/\kappa'}
 =r|B|^{-1}.
\]
This proves \eqref{E.atomsup}.
\end{proof}

\section{Proof of Theorem \ref{T.main}}\label{S.main}

Throughout this section the constants may depend on $n$, $\ell$,
$d_\nu^{-1}$, $\|\dhalf\phi\|_{\BMOpar}$ and the integrability exponent,
unless a more restricted dependence is stated.

\subsection{The two maps and their composite}\label{ss.maps}

Put $\gamma=c(n)/(1+\ell)$, with $c(n)>0$ small enough that
\begin{equation}\label{E.gamma}
  \mathsf J:=\pd_{x_n}\Big(x_n+P_{\gamma x_n}\phi\Big)
   =1+\gamma\br{\pd_\lambda P_\lambda\phi}\big|_{\lambda=\gamma x_n}
  \ \in\ \big[\tfrac12,\tfrac32\big] ,
\end{equation}
which the bound $\abs{\pd_\lambda P_\lambda\phi}\le C(n)\ell$ of
\S\ref{ss.moll} allows, and let
\begin{equation}\label{E.LM}
  \Pi(x,x_n,t):=\br{x,\ x_n+P_{\gamma x_n}\phi(x,t),\ t} .
\end{equation}
Write $\Pi(X,t)=(\Theta(X,t),t)$. This is the Hofmann--Lewis map
\cite[(1.6)]{HLmem}. It is a bijection of $\Uc$ onto $\om$, bi-Lipschitz
for the parabolic distance with a constant depending only on $n$ and $\ell$, and
$\Pi(x,0,t)=\pi(x,t)$. Let $\Pi_\nu$ be the map \eqref{E.LM} built from
$\phi_\nu$ in the coordinate system of $\nu$. We choose this spatial frame
to be positively oriented, which leaves the boundary norms unchanged.
By \eqref{E.graphbound} we may take
\[
 \ell_\nu:=\frac{\sqrt{1+\ell_0^2}+\ell}{d_\nu},\qquad
 \gamma_\nu:=\frac{c(n)}{1+\ell_\nu}.
\]
The analogue of \eqref{E.gamma} then holds for $\Pi_\nu$, and its
bi-Lipschitz constant is controlled by $n$, $\ell$ and $d_\nu^{-1}$.
Write $\Pi_\nu(X,t)=(\Theta_\nu(X,t),t)$. Thus
$\Pi_\nu(x,0,t)=\pi_\nu(x,t)$. Set
\begin{equation}\label{E.T}
  \widehat T:=\Pi^{-1}\circ\Pi_\nu,\qquad
  \widehat T(X,t)=\br{T(X,t),t},\qquad
  M:=D_XT,\qquad \varrho:=\det M .
\end{equation}

\begin{lemma}\label{L.T}
$\widehat T$ is a bijection of $\Uc$ onto itself which preserves the time
variable. It is
bi-Lipschitz for the parabolic distance with a constant $\mathsf K$ depending
only on $n$, $\ell$ and $d_\nu^{-1}$, and satisfies
\begin{equation}\label{E.Tbdry}
  \br{T(x,0,t),t}=\br{\pi^{-1}\circ\pi_\nu}(x,t),\qquad
  \varrho\approx1,\qquad
  \mathsf K^{-1}\le\frac{(T(X,t))_n}{x_n}\le\mathsf K .
\end{equation}
Moreover $M=\br{J\circ\widehat T}^{-1}J_\nu$, where
$J:=D_X\Theta$ and $J_\nu:=D_X\Theta_\nu$, and
\begin{equation}\label{E.dtT}
  \br{J\circ\widehat T}\br{\pd_tT}
  \ =\ \pd_t\Theta_\nu-\br{\pd_t\Theta}\circ\widehat T .
\end{equation}
\end{lemma}

\begin{proof}
Both $\Pi$ and $\Pi_\nu$ are bi-Lipschitz bijections of $\Uc$ onto $\om$ taking
$\pd\Uc$ onto $\pom$, by \eqref{E.gamma} and \eqref{E.mollbds}, so the same
holds for $\widehat T$, and \eqref{E.Tbdry} follows. The identity for $M$ is the chain
rule, and \eqref{E.dtT} is obtained by differentiating
$\Theta\br{T(X,t),t}=\Theta_\nu(X,t)$ in $t$.
\end{proof}

\noindent The map $\widehat T$ is a bi-Lipschitz extension into $\Uc$ of the transition
map between the two graph parametrisations of $\pom$.

\subsection{The PDE for $w$}\label{ss.pde}

\begin{proposition}\label{P.pde}
Let $u$ be an energy solution of $\pd_tu=\Delta u$ in $\Uc$ and put
$w:=u\circ\widehat T$, with $\widehat T$, $T$, $M$ and $\varrho$ as in
\eqref{E.T}. Then $w$ is a weak solution in $\Uc$ of
\begin{equation}\label{E.Mw}
  -\pd_tw+\divg_X\br{B\nabla_Xw}+\mathbf g\cdot\nabla_Xw\ =\ 0 ,
\end{equation}
where
\begin{equation}\label{E.Bg}
  B:=M^{-1}M^{-\mathsf T},\qquad
  \mathbf g:=\varrho^{-1}B\nabla_X\varrho+M^{-1}\pd_tT .
\end{equation}
\end{proposition}

\begin{proof}
Let $\Psi\in C^\infty_0(\Uc)$ be a test function for $w$, so that
$\Psi\circ\widehat T^{-1}$ is the corresponding test function for $u$. From
$w=u\circ\widehat T$,
$$\nabla_Xw=M^{\mathsf T}\br{\nabla_Xu}\circ\widehat T,\qquad
  \pd_tw=\br{\pd_tu}\circ\widehat T
  +\br{\pd_tT}\cdot\br{\nabla_Xu}\circ\widehat T .$$
Inserting $\Psi\circ\widehat T^{-1}$ in the weak formulation of
$\pd_tu=\Delta u$, changing
variables by the spatial map $T(\cdot,t)$ at each fixed $t$, and using
$\br{M^{-\mathsf T}a}\cdot\br{M^{-\mathsf T}b}
=\br{M^{-1}M^{-\mathsf T}a}\cdot b$, one gets
$$\int_\R\int_\Uc\Big[\br{\varrho B\nabla_Xw}\cdot\nabla_X\Psi
  +\varrho\br{\pd_tw}\Psi
  -\varrho\br{M^{-1}\pd_tT}\cdot\nabla_Xw\,\Psi\Big]\d X\d t=0 ,$$
that is,
$$-\varrho\pd_tw+\divg_X\br{\varrho B\nabla_Xw}
  +\varrho\br{M^{-1}\pd_tT}\cdot\nabla_Xw\ =\ 0 .$$
Dividing by $\varrho$ and using
$$\varrho^{-1}\divg_X\br{\varrho B\nabla_Xw}
  =\divg_X\br{B\nabla_Xw}+\varrho^{-1}\br{B\nabla_X\varrho}\cdot\nabla_Xw$$
gives \eqref{E.Mw} with \eqref{E.Bg}.
\end{proof}

\subsection{The coefficients of the PDE for $w$}\label{ss.coeff}

\begin{proposition}\label{P.coeff}
With constants depending only on $n$, $\ell$, $d_\nu^{-1}$ and
$\|\dhalf\phi\|_{\BMOpar}$, the matrix $B$ is uniformly elliptic,
\begin{equation}\label{E.ptwise}
  x_n\br{\abs{\nabla_XM}+\abs{\nabla_XB}+\abs{\mathbf g}}
  +x_n^2\br{\abs{\pd_tB}+\abs{\nabla_X\mathbf g}}
  +x_n^3\abs{\pd_t\mathbf g}\ \lesssim\ 1 ,
\end{equation}
and
\begin{equation}\label{E.carl}
  \abs{\nabla_XM}^2x_n+\abs{\nabla_XB}^2x_n
  +\abs{\pd_tB}^2x_n^3+\abs{\mathbf g}^2x_n
  +\abs{\nabla_X\mathbf g}^2x_n^3
  +\abs{\pd_t\mathbf g}^2x_n^5
\end{equation}
is a Carleson density on $\Uc$, as is its supremum over the Whitney ball
$B_{x_n/2}(X,t)$.
\end{proposition}

\noindent Thus \eqref{E.ptwise} and \eqref{E.carl} give all the coefficient
and drift hypotheses of Theorems \ref{T.SN} and \ref{T.half}, with constants
and Carleson norms of the stated dependence.

\begin{proof}
Uniform ellipticity follows from the bounds for $M$ and $M^{-1}$ in
Lemma \ref{L.T}.

Every entry of $J$ or $J_\nu$ is $1$, $0$, a component of the corresponding
mollified spatial gradient, or the corresponding vertical derivative. The
determinants are bounded away from zero by \eqref{E.gamma} and its analogue
for $\Pi_\nu$. Hence every entry
of $M$ and $B$, as well as $\varrho$, is obtained by applying a smooth
function, whose derivatives are bounded on the relevant compact range, to
\begin{equation}\label{E.list}
  P_{\gamma x_n}\br{\nabla_y\phi},\quad
  \br{\pd_\lambda P_\lambda\phi}\big|_{\gamma x_n},\quad
  P_{\gamma_\nu x_n}\br{\nabla_y\phi_\nu},\quad
  \br{\pd_\lambda P_\lambda\phi_\nu}\big|_{\gamma_\nu x_n} ,
\end{equation}
the first two quantities being evaluated at $\widehat T(X,t)$ and the last two
at $(X,t)$. Every entry of \eqref{E.list} is a first-order derivative
$\pd^{\,1}P_{\gamma x_n}\phi$, or the corresponding derivative of
$\phi_\nu$. Hence a derivative $\nabla_X$ of such an entry is an
instance of $\pd^{\,2}$ with $\theta=0$, and a derivative $\pd_t$ of it an
instance of $\pd^{\,2}$ with $\theta=1$. The chain rule and the instances
$(l,\theta)=(2,0)$ and $(2,1)$ of Lemma \ref{L.HL} therefore give
\eqref{E.ptwise} for $M$, $B$ and $\varrho$, and make
$$\abs{\nabla_XM}^2x_n+\abs{\nabla_XB}^2x_n
  +\abs{\pd_tB}^2x_n^3$$
Carleson, the exponents $2l+2\theta-3$ of \eqref{E.mollcarl} being $1$ and $3$
for these two cases. By \eqref{E.Tbdry}, Carleson densities are preserved by
the bi-Lipschitz map $\widehat T$, so either evaluation point in
\eqref{E.list} is admissible.\medskip

For $\mathbf g$, the summand $\varrho^{-1}B\nabla_X\varrho$ is covered by the
previous paragraph. For the summand $M^{-1}\pd_tT$, \eqref{E.dtT} gives
$$\abs{\pd_tT}\ \lesssim\ \abs{\pd_tP_{\gamma_\nu x_n}\phi_\nu}
  +\abs{\pd_tP_{\gamma x_n}\phi}\circ\widehat T ,$$
and both summands are the instance $l=\theta=1$ of Lemma \ref{L.HL}: they are
$\lesssim x_n^{-1}$, and $\abs{\pd_tP_{\gamma x_n}\phi}^2x_n$ is Carleson, the
exponent $2l+2\theta-3$ being $1$. For $\phi_\nu$ this uses
$\|\dhalf\phi_\nu\|_{\BMOpar}<\infty$, which is Corollary \ref{C.BHMN}. This is
the only place where the Lewis--Murray condition is used, and it is used in both
coordinate systems.\medskip

Differentiating \eqref{E.Bg} and \eqref{E.dtT} shows that every summand of
$\nabla_X\mathbf g$ and $\pd_t\mathbf g$ is a product of derivatives of the
entries in \eqref{E.list} of total parabolic order two and three,
respectively. Here a spatial derivative has order one and a time derivative
has order two. If one factor has positive order $q$, its square times
$x_n^{2q-1}$ is Carleson by Lemma \ref{L.HL}. The remaining factors are
bounded pointwise by the reciprocal power of $x_n$. Hence the square of a
product of total order two times $x_n^3$, and of total order three times
$x_n^5$, is Carleson. The pointwise part of Lemma \ref{L.HL} gives the
corresponding assertions in \eqref{E.ptwise}.

Finally, consider a component $H$ of
$\nabla_XM$, $\nabla_XB$, $\pd_tB$, $\mathbf g$,
$\nabla_X\mathbf g$ or $\pd_t\mathbf g$, with respective parabolic order
$q=1,1,2,1,2,3$. A mollifier derivative
$\pd^{\,l}P_{\gamma x_n}a$ has order $l+\theta-1$, so bounded first
spatial derivatives have order zero. The chain-rule argument above gives,
for every multi-index $\alpha$ and integer $j\ge0$, the Carleson density
\[
 m_{\alpha,j}(Y,s):=
 |\pd_Y^\alpha\pd_s^jH(Y,s)|^2y_n^{2(q+|\alpha|+2j)-1}.
\]
In each differentiated product one positive-order mollifier derivative is
kept in its Carleson density, while the other factors are bounded pointwise
with the remaining powers of $y_n$. Composed factors are handled by
$\widehat T$, and Lemma \ref{L.HL} holds at every order.\medskip

\noindent Parabolic rescaling and Sobolev embedding on
$B_{3x_n/4}(X,t)$, with an integer $N>(n+1)/2$, give
\[
 x_n^{2q-1}\sup_{B_{x_n/2}(X,t)}|H|^2
 \lesssim\sum_{|\alpha|+j\le N}
       \fiint_{B_{3x_n/4}(X,t)}m_{\alpha,j}(Y,s)\,\d Y\d s,
\]
because $x_n/4<y_n<7x_n/4$ on the averaging ball.
Whitney averaging preserves Carleson densities. Indeed, for $m\ge0$ put
\[
 \begin{gathered}
 \mathcal Wm(X,t):=\fiint_{B_{3x_n/4}(X,t)}m(Y,s)\,\d Y\d s,\\
 \mathcal T_R(z_0):=\{(x,x_n,t):(x,t)\in B_R(z_0),\ 0<x_n<R\}.
 \end{gathered}
\]
If $(X,t)\in\mathcal T_R(z_0)$ and $(Y,s)\in B_{3x_n/4}(X,t)$, then
$(Y,s)\in\mathcal T_{2R}(z_0)$ and $y_n\approx x_n$. For fixed $(Y,s)$
the eligible centres have volume at most $Cy_n^{n+2}$, comparable to the
averaging denominator. Fubini therefore gives
\[
 \iint_{\mathcal T_R(z_0)}\mathcal Wm
 \lesssim\iint_{\mathcal T_{2R}(z_0)}m
 \lesssim\Cnorm{m\,\d X\d t}\,|B_R|.
\]
Apply this to each $m_{\alpha,j}$ and sum over the components $H$.
Since the height is comparable throughout each Whitney ball, this proves
the supremum assertion for \eqref{E.carl}.
\end{proof}

\subsection{The nontangential bound and the trace}\label{ss.trace}

From now on $F$ is measurable on $\pom$ with $\|F\|_{p,e_n}<\infty$, $u$ is the
caloric extension of $f:=F\circ\pi$ to $\Uc$, and $w:=u\circ\widehat T$, which
solves \eqref{E.Mw} by Proposition \ref{P.pde}.

\begin{proposition}\label{P.Nw}
For every $1<p<\infty$, $w=F\circ\pi_\nu$ on $\pd\Uc$ and
\begin{equation}\label{E.Nw}
  \big\|N(\nabla_Xw)\big\|_{L^p(\R^{n-1}\times\R)}\ \lesssim\ \|F\|_{p,e_n} .
\end{equation}
\end{proposition}

\begin{proof}
The first assertion follows from
$\br{T(x,0,t),t}=\br{\pi^{-1}\circ\pi_\nu}(x,t)$, which is
\eqref{E.Tbdry}, together with $u=F\circ\pi$ on $\pd\Uc$. Next,
$\nabla_Xw=M^{\mathsf T}(\nabla_Xu)\circ\widehat T$
with $M$ and $M^{-1}$ bounded, so
$\abs{\nabla_Xw}\approx\abs{\nabla_Xu}\circ\widehat T$. Moreover $\widehat T$ is a parabolically
bi-Lipschitz bijection of $\Uc$ fixing $\pd\Uc$
setwise, so it takes nontangential cones into nontangential cones of a
comparable aperture in both directions. Its boundary restriction preserves
Lebesgue measure up to constants. Hence
$N(\nabla_Xw)\approx N(\nabla_Xu)\circ\br{\widehat T|_{\pd\Uc}}$ up to a change of
aperture of nontangential cones, and hence \eqref{E.Nw} follows from
\eqref{E.flat} of Proposition \ref{P.flat}.
\end{proof}

\begin{proposition}\label{P.Sw}
Let $1<p<\infty$, and suppose that
$S(\nabla_Xu)\in L^p(\R^{n-1}\times\R)$. Then
\begin{equation}\label{E.SNpull}
  \big\|S(\nabla_Xw)\big\|_{L^p}
  \ \lesssim\ \big\|N(\nabla_Xw)\big\|_{L^p} .
\end{equation}
\end{proposition}

\begin{proof}
By \eqref{E.functionals} the square function $S(\nabla_Xw)$ is the conical
integral of $\abs{\nabla_X^2w}^2\delta^{-n}$. Let
$\Lambda:=\widehat T|_{\pd\Uc}$.
Differentiating $\nabla_Xw=M^{\mathsf T}(\nabla_Xu)\circ\widehat T$ in the spatial
variables gives
$$
  \abs{\nabla_X^2w}\ \lesssim\
  \abs{(\nabla_X^2u)\circ\widehat T}
  +\abs{\nabla_XM}\,\abs{(\nabla_Xu)\circ\widehat T} .
$$
The map $\widehat T$ takes cones into cones of a comparable aperture. Moreover,
$$
  \varrho\approx1,\qquad (T(X,t))_n\approx x_n .
$$
Hence the first summand has square function bounded by
$$
  S_{a'}(\nabla_Xu)\circ\Lambda
$$
for a fixed larger aperture $a'$. By Theorem \ref{T.SN} with $A=I$, and by
the change of variables on the boundary,
$$
  \big\|S_{a'}(\nabla_Xu)\circ\Lambda\big\|_{L^p}
  \ \lesssim\ \big\|N(\nabla_Xu)\big\|_{L^p} .
$$

For the second summand let
$$
  \d\mu_M:=\abs{\nabla_XM}^2x_n\,\d X\d t .
$$
This measure is Carleson by Proposition \ref{P.coeff}. Its contribution is
the conical square function
$$
 \left(\iint_{\Gam(q,\tau)}
   \abs{(\nabla_Xu)\circ\widehat T}^2\,
   \frac{\d\mu_M}{x_n^{n+1}}\right)^{1/2}.
$$
The cone distortion gives, for a sufficiently large $a''$,
$$
  \abs{(\nabla_Xu)(T(X,t),t)}\ \le\
  \inf_{(y,s)\in B_{c x_n}(x,t)}
  \left[N_{a''}(\nabla_Xu)\circ\Lambda\right](y,s).
$$
The conical Carleson embedding \cite[Theorem 3]{CMS} therefore bounds the
$L^p$ norm of this square function by
$$
  C\Cnorm{\mu_M}^{1/2}
  \big\|N_{a''}(\nabla_Xu)\circ\Lambda\big\|_{L^p}.
$$
The proof of Proposition \ref{P.Nw}, applied with the changed aperture,
shows that the last norm is comparable to
$\|N(\nabla_Xw)\|_{L^p}$. The same comparison applies to the first
summand, and gives \eqref{E.SNpull}.
\end{proof}

\begin{theorem}\label{T.trace}
Let $1<p<\infty$. Then $F\circ\pi_\nu\in\dot L^p_{1,1/2}(\R^{n-1}\times\R)$ and
\begin{equation}\label{E.trace}
  \|F\|_{p,\nu}\ \le\ C\big\|N(\nabla_Xw)\big\|_{L^p}\ \le\ C\,\|F\|_{p,e_n} .
\end{equation}
\end{theorem}

\begin{proof}
The second inequality is Proposition \ref{P.Nw}. For the first one we bound
separately the two halves of the right-hand side of \eqref{E.norms}, applied to
$F\circ\pi_\nu$.

\noindent We may assume that $f:=F\circ\pi$ has
$\wh f\in C^\infty_0(\R^n\setminus\set0)$. Such functions are dense in
$\dot L^p_{1,1/2}$. Indeed, choose $f_j$ of this form with
$f_j\to f$ in the homogeneous norm. Since the space is homogeneous, subtract
constants so that the averages of $f_j$ and $f$ over one fixed parabolic ball
agree. The parabolic Poincar\'e inequality then gives
$f_j\to f$ in $L^p_{\loc}$.

Since
$\Lambda:=\pi^{-1}\circ\pi_\nu$ and its inverse are bi-Lipschitz and preserve
Lebesgue measure up to constants, $f_j\circ\Lambda\to f\circ\Lambda$ in
$L^p_{\loc}$ and hence in $\mathcal D'$. The added constants do not affect any
quantity in \eqref{E.trace}. Thus \eqref{E.trace} for the $f_j$ passes to the
limit by weak lower semicontinuity of the two $L^p$ norms on its left. This is
the only role the reduced class plays: the estimate is proved on it and
extended by continuity.\medskip

\noindent Below we use the frequency-localised representative before this
normalisation. Replacing it by the normalised representative changes its
caloric extension, and hence $w$, only by the same constant, so all derivatives,
differences and estimates below are unchanged.

\noindent The reduction puts $w$ in the energy class required by Theorem
\ref{T.half}. For such an $f$ the caloric
extension is an energy solution in the sense of \cite[\S2]{DLPN}, since
$\wh u=\wh f\,e^{-x_n\sqrt{\abs\xi^2+i\tau}}$ with $\wh f$ supported in
$c\le\|(\xi,\tau)\|\le C$, so that $u(\cdot,x_n,\cdot)$ is Schwartz with every
seminorm $O(e^{-cx_n})$.

\noindent The same is true of $w$. Its spatial energy is controlled by
$\varrho\approx1$ and
$\abs{\nabla_Xw}\approx\abs{\nabla_Xu}\circ\widehat T$.
For the half-time derivative we use the Gagliardo form of the norm and split
$$w(X,t)-w(X,s)=\Big[u\br{T(X,t),t}-u\br{T(X,s),t}\Big]
  +\Big[u\br{T(X,s),t}-u\br{T(X,s),s}\Big] .$$
The second bracket contributes $\|\dhalf u\|^2_{L^2(\Uc)}$, by the change of
variables $X\mapsto T(X,s)$ at each fixed $s$. For the first, Lemma \ref{L.T}
and \eqref{E.mollbds} give
\begin{equation}\label{E.Tdisp}
  \abs{T(X,t)-T(X,s)}\ \lesssim\
  \min\set{\abs{t-s}/x_n,\ \abs{t-s}^{1/2}} ,
\end{equation}
and the square of \eqref{E.Tdisp} divided by $\abs{t-s}^2$ integrates over
$\abs{t-s}\le1$ to $\lesssim1+\log(1/x_n)$. For $\abs{t-s}\ge1$ we bound the
difference by $\abs u$ itself. The resulting terms are absorbed by the decay
of $u$. Hence $\dhalf w\in L^2(\Uc)$; the transported term contributes one
logarithm.\medskip

\noindent\emph{Step 1: the spatial gradient.} For each $x_n>0$ the function
$\nabla_yw(\cdot,x_n,\cdot)$ is bounded in $L^p(\R^{n-1}\times\R)$ by
$\|N(\nabla_Xw)\|_{L^p}$, and $w(\cdot,x_n,\cdot)\to F\circ\pi_\nu$ in
$\mathcal D'(\R^{n-1}\times\R)$ as $x_n\to0$. Hence
$\nabla_y(F\circ\pi_\nu)$ is the weak limit of a bounded family, and the $L^p$
norm being weakly lower semicontinuous,
$$\big\|\nabla_y\br{F\circ\pi_\nu}\big\|_{L^p}\ \le\
  \big\|N(\nabla_Xw)\big\|_{L^p} .$$
No PDE is used here.\medskip

\noindent\emph{Step 2: the square function.} For the reduced class, the
Fourier representation of $u$ and the support of $\wh f$ imply
$S(\nabla_Xu)\in L^p$. Proposition \ref{P.Sw} therefore gives
\begin{equation}\label{E.SN}
  \big\|S(\nabla_Xw)\big\|_{L^p}\ \lesssim\ \big\|N(\nabla_Xw)\big\|_{L^p} .
\end{equation}

\noindent\emph{Step 3: the half derivative.} Since
$\operatorname{Tr}w=F\circ\pi_\nu$ and $\dhalf$ acts in the time variable,
$$\dhalf\br{F\circ\pi_\nu}=\dhalf\operatorname{Tr}w.$$
Theorem \ref{T.half}, followed by \eqref{E.SN} and $\Nt\le N_{a'}$, gives
$$\big\|\dhalf\br{F\circ\pi_\nu}\big\|_{L^p}\ \lesssim\
  \big\|\Nt(\nabla_Xw)\big\|_{L^p}+\big\|S(\nabla_Xw)\big\|_{L^p}
  \ \lesssim\ \big\|N(\nabla_Xw)\big\|_{L^p} .$$
Adding the estimates of Steps 1 and 3 gives \eqref{E.trace}.
\end{proof}

\begin{corollary}\label{C.frame}
Theorem \ref{T.main} holds for every $1<p<\infty$.
\end{corollary}

\begin{proof}
Theorem \ref{T.trace} is the inequality $\|F\|_{p,\nu}\lesssim\|F\|_{p,e_n}$.
Of the two coordinate systems, the construction above uses only that $\pom$ is
the graph of $\phi$ in the direction $e_n$ and the graph of $\phi_\nu$ in the
direction $\nu$, that both graph functions are $\Lip(1,\tfrac12)$ with
characters controlled by $\ell$ and $d_\nu^{-1}$ by Lemma
\ref{L.graphgeometry}, and that both have
half-derivative in $\BMOpar$, which is Corollary \ref{C.BHMN}. These data are
symmetric in the two graph functions. We repeat the argument with
$\widehat T^{-1}=\Pi_\nu^{-1}\circ\Pi$, whose quantitative bounds follow from
the same two maps. Thus no new angle criterion is needed, and we obtain
$\|F\|_{p,e_n}\lesssim\|F\|_{p,\nu}$.
\end{proof}

\section{Coordinate independence fails without the Lewis--Murray condition}
\label{S.counter}

The spatial half of the boundary norm never depends on the coordinate system,
whatever $\phi$ is. We record this first, since it locates the failure in the
half-derivative.

\begin{lemma}[the spatial gradient is frame-free]\label{L.spatial}
Let $\phi$ be $\Lip(1,\tfrac12)$, let $\nu$ be admissible and let
$\Lambda:=\pi^{-1}\circ\pi_\nu$. Then for every $1\le p\le\infty$ and every
measurable $F$ on $\pom$,
\begin{equation}\label{E.spatial}
  \big\|\nabla_y\br{F\circ\pi_\nu}\big\|_{L^p(\R^{n-1}\times\R)}
  \ \approx\ \big\|\nabla_y\br{F\circ\pi}\big\|_{L^p(\R^{n-1}\times\R)} ,
\end{equation}
with constants depending only on $n$, $\ell_0$ and $d_\nu^{-1}$. No
hypothesis on $\dhalf\phi$ is used.
For a measurable pull-back, its gradient norm is understood as $+\infty$
unless its time slices belong to $W^{1,1}_{\loc}(\R^{n-1})$ for almost every
time and its spatial weak gradient belongs to the indicated $L^p$ space.
\end{lemma}

\begin{proof}
The map $\Lambda$ preserves the time variable, and we write
$\Lambda(y,t)=(\lambda_t(y),t)$ as in Lemma \ref{L.graphgeometry}.
That lemma bounds the Lipschitz constants of $\lambda_t$ and its inverse
in terms of $\ell_0$ and $d_\nu^{-1}$, uniformly in $t$.
The Jacobian bounds are \eqref{E.graphjac}.\medskip

Clearly, $F\circ\pi_\nu(\cdot,t)=\br{F\circ\pi(\cdot,t)}\circ\lambda_t$, so that
$$\nabla_y\br{F\circ\pi_\nu}(\cdot,t)
  \ =\ \br{D\lambda_t}^{\mathsf T}
   \Big(\nabla_y\br{F\circ\pi}\Big)\circ\lambda_t
  \qquad\text{almost everywhere},$$
and $D\lambda_t$ and its inverse are bounded, whence
$\abs{\nabla_y(F\circ\pi_\nu)}\approx\abs{\nabla_y(F\circ\pi)}\circ\lambda_t$
pointwise. For $p=\infty$ this is \eqref{E.spatial}. For $p<\infty$, changing
variables by $\lambda_t$ at each fixed $t$ and integrating in $t$ gives
\eqref{E.spatial}, the Jacobian being comparable to $1$.
Applying the same argument to $\lambda_t^{-1}$ shows that finiteness on
either side implies finiteness on the other, which also proves the
extended-value assertion.
\end{proof}

\noindent Theorem \ref{T.counter} is Proposition \ref{P.counter} below
together with Lemma \ref{L.spatial}.

\begin{proposition}[the boundary norm depends on the representation]\label{P.counter}
Assume $n\ge2$. Let
\[
  b(t):=\sum_{k\ge1}2^{-k/2}\sin(2^{k}t) ,
\]
a bounded $2\pi$-periodic function of class $C^{1/2}(\R)$ with
$b\notin H^{1/2}_{\loc}$. Let $\varepsilon>0$ and $\gamma\in(0,\pi/2)$ be arbitrary,
put $\psi(x,t):=\varepsilon b(t)$ and $\om_\psi:=\set{x_n>\psi(x,t)}$, and let
$\nu:=\cos\gamma\,e_n+\sin\gamma\,e_1$. Choose
$F_0\in C_c^\infty(\R)$, $g_0\in C_c^\infty(\R^{n-2})$ and
$\chi\in C_c^\infty(\R)$, none identically zero, with the convention
$g_0\equiv1$ when $n=2$, and let $f$ be the
function on $\pd\om_\psi$ given by
\[
  f\br{x,\psi(x,t),t}=F_0(x_1)\,g_0(x_2,\dots,x_{n-1})\,\chi(t) .
\]
Then $\|f\|_{p,e_n}<\infty$ for every $1<p<\infty$, that is, $f$ lies in
$\dot L^p_{1,1/2}(\pd\om_\psi)$ computed in the coordinate system whose
vertical direction is $e_n$. But $\|f\|_{2,\nu}=\infty$, that is, $f$ does not
lie in $\dot L^2_{1,1/2}(\pd\om_\psi)$ computed in the coordinate system whose
vertical direction is $\nu$.
\end{proposition}

The set $\om_\psi$ lies above a $\Lip(1,\tfrac12)$ graph, and
$\dhalf b\notin\BMO$. That is what the proposition exploits. Indeed $b$ is
periodic with $\sum_m\abs m\abs{\wh b(m)}^2=\infty$, so
$\dhalf b\notin L^2_{\loc}(\R)$, while $\BMO\subset L^2_{\loc}(\R)$.

Observe that in the coordinates of $e_n$ the boundary is flat in space and
varies in time only. Hence its spatial Lipschitz constant is $0$, and every
$\nu$ with $\angle(\nu,e_n)<\pi/2$ is an admissible direction.

\begin{proof}
\emph{The function $b$.} Boundedness and periodicity are immediate. To verify
the $C^{1/2}$ estimate, given $0<|h|\le1$ choose $K$ with
$2^{-K}\le|h|<2^{-K+1}$ and split the sum at $k=K$, using
$|\sin(2^k(t+h))-\sin(2^kt)|\le2^k|h|$ for $k\le K$ and $|\sin|\le1$ for
$k>K$. This gives
$$|b(t+h)-b(t)|\ \lesssim\ |h|2^{K/2}+2^{-K/2}\ \lesssim\ |h|^{1/2} .$$
On the torus $\wh b(\pm2^k)=\mp\frac i22^{-k/2}$ and all other coefficients
vanish, so $\sum_m|m||\wh b(m)|^2=\infty$. We use the sharper statement
that for every $\eta_0>0$
\begin{equation}\label{E.bdiv}
  \int_{|u|\le\eta_0}\int_0^{2\pi}\frac{|b(t)-b(t-u)|^2}{u^2}\d t\d u=\infty .
\end{equation}
Indeed by Parseval, for every $u$,
\[
  \rho(u):=\int_0^{2\pi}|b(t)-b(t-u)|^2\d t
  =8\pi\sum_m|\wh b(m)|^2\sin^2(mu/2) .
\]
The substitution $v=mu/2$ gives
$\int_{|u|\le\eta_0}\sin^2(mu/2)u^{-2}\d u\gtrsim|m|$
for $|m|\ge\pi/\eta_0$. Tonelli's theorem and the divergence of
$\sum_m|m||\wh b(m)|^2$ prove \eqref{E.bdiv}.
In fact, for every nonempty bounded open interval $I$ and every $\eta>0$,
\begin{equation}\label{E.bdivlocal}
 \iint_{\substack{t,s\in I\\0<|t-s|<\eta}}
 \frac{|b(t)-b(s)|^2}{|t-s|^2}\d t\d s=\infty.
\end{equation}
To see this, choose an integer $K\ge1$ so large that $2\pi2^{-K}<|I|$,
and write
\[
 b(t)=b_K(t)+r_K(t),\qquad
 b_K(t):=\sum_{k=1}^K2^{-k/2}\sin(2^kt),\qquad
 r_K(t)=2^{-K/2}b(2^Kt).
\]
Choose a closed interval $J\subset I$ of length $2\pi2^{-K}$ and
$0<\eta_1<\min\{\eta,\operatorname{dist}(J,\R\setminus I)\}$.
By periodicity, rescaling \eqref{E.bdiv} gives
\[
 \int_{|u|<\eta_1}\int_J
 \frac{|r_K(t)-r_K(t-u)|^2}{u^2}\d t\d u=\infty.
\]
Both $t$ and $t-u$ lie in $I$. Since $b_K$ is smooth with bounded
derivative, its corresponding integral is finite. The inequality
$|r_K(t)-r_K(s)|^2\le2|b(t)-b(s)|^2+2|b_K(t)-b_K(s)|^2$
therefore proves \eqref{E.bdivlocal}.\vglue1mm

\emph{The two coordinate systems.} Let
$e_1'=\cos\gamma\,e_1-\sin\gamma\,e_n$ and $e_j'=e_j$ for
$2\le j\le n-1$, so that these vectors span $\nu^\perp$. Writing
$X=\sum_jy_je_j'+y_n\nu$ gives $X\cdot e_n=-y_1\sin\gamma+y_n\cos\gamma$, so
$X\in(\om_\psi)_t$ if and only if
$y_n>\varphi_\nu(y,t):=\varepsilon b(t)/\cos\gamma+y_1\tan\gamma$. Expanding
$X=\sum_jy_je_j'+\varphi_\nu\nu$ in the standard basis gives
$$X\cdot e_1=\frac{y_1}{\cos\gamma}+\varepsilon b(t)\tan\gamma .$$
Hence, writing $F$ and $G$ for
the pull-backs of $f$ in the two coordinate systems and $\varkappa:=\varepsilon\tan\gamma$,
\begin{equation}\label{E.framecx}
  G(y,t)=F\Big(\frac{y_1}{\cos\gamma}+\varkappa b(t),\ y_2,\dots,y_{n-1},\ t\Big),
\end{equation}
a translation along a $C^{1/2}$ path. The norm in the $e_n$ coordinates is
finite for every $1<p<\infty$ because
$F(x,t)=F_0(x_1)g_0(x_2,\dots,x_{n-1})\chi(t)$ is smooth and compactly
supported. Lemma \ref{L.spatial} also gives the finiteness of the spatial
gradient norm in the $\nu$ coordinates.\vglue1mm

\emph{The divergence.} Since $\chi$ is not identically zero, there are a
nonempty bounded open interval $I$ and $c_\chi>0$ such that
$|\chi(t)|\ge c_\chi$ on $I$. Let $\mathsf T_hF_0(z):=F_0(z+h)$.
Since $h^{-1}(\mathsf T_hF_0-F_0)\to F_0'$ in $L^2(\R)$ as $h\to0$ and
$\|F_0'\|_{L^2}>0$, there are $h_0,c_2>0$ such that
\[
 \|\mathsf T_hF_0-F_0\|_{L^2}^2\ge c_2h^2
 \qquad\text{whenever }|h|\le h_0.
\]
For $t,s\in I$ put $h:=\varkappa(b(t)-b(s))$. Equation \eqref{E.framecx}
and the change of variable $z=y_1/\cos\gamma+\varkappa b(s)$ give
\[
 \int_{\R^{n-1}}|G(y,t)-G(y,s)|^2\d y
 =\|g_0\|_{L^2}^2\cos\gamma\,
   \|\chi(t)\mathsf T_hF_0-\chi(s)F_0\|_{L^2}^2,
\]
where $\|g_0\|_{L^2}^2$ is read as $1$ when $n=2$. Decomposing the last
difference as
\[
 \chi(t)(\mathsf T_hF_0-F_0)+(\chi(t)-\chi(s))F_0
\]
and using $\|A+B\|_2^2\ge\tfrac12\|A\|_2^2-\|B\|_2^2$, we obtain
\[
 \int_{\R^{n-1}}|G(y,t)-G(y,s)|^2\d y
 \ \ge\ c\varkappa^2|b(t)-b(s)|^2-C|t-s|^2
\]
whenever $|h|\le h_0$, with $c>0$. Here we used
$|\chi(t)-\chi(s)|\le\|\chi'\|_\infty|t-s|$.
The $C^{1/2}$ bound for $b$ gives $\eta_0>0$ such that $|h|\le h_0$
if $|t-s|<\eta_0$.

Divide the last estimate by $|t-s|^2$ and integrate over
$t,s\in I$ with $\delta<|t-s|<\eta_0$. The error is bounded by
$C|I|^2$, independently of $\delta>0$, whereas \eqref{E.bdivlocal} makes
the positive term diverge as $\delta\downarrow0$. Since $G\in L^2$, the
Gagliardo identity, valid also with value $+\infty$, gives
\[
 \|\dhalf G\|_{L^2(\R^n)}^2
 =\frac1{2\pi}\iint_{\R^2}
   \frac{\|G(\cdot,t)-G(\cdot,s)\|_{L^2(\R^{n-1})}^2}{|t-s|^2}
   \d t\d s=\infty.\qedhere
\]
\end{proof}

\section{The local estimate}\label{S.loc}

This section proves Theorem \ref{T.locest}. An estimate of this kind cannot
follow from interior estimates alone. For a solution vanishing on a boundary
portion, the Caccioppoli inequality and boundary H\"older continuity give
$\abs{\nabla u}\lesssim\delta^{-1}(\delta/r)^\beta\sup\abs u$ on a Whitney box
at height $\delta$, and $\delta^{\beta-1}$ diverges as $\delta\to0$. Our
argument is inspired by the proof of the Neumann counterpart in
\cite[\S\S3--4]{DLP2}.

\subsection{The Poisson--Dirichlet problem}\label{ss.PD}

We need to consider an  inhomogeneous problem, this is  the only place in the paper
where a solution is not given by \eqref{E.pmsol}. For
$F\in L^\infty_c(\om;\R^n)$ we call $v$ the \emph{finite-energy solution} of
$\LL^*v=\divg F$ with zero boundary data if $\nabla v\in L^2(\om)$, if
$v(\cdot,s)\in H^1_0(\om_s)$ for almost every $s$, if $v$ vanishes for $t$ above
the support of $F$, 
$$\sup_s\|v(\cdot,s)\|_{L^2(\om_s)}<\infty,$$
and
\begin{equation}\label{E.PDweak}
  \iint_\om\br{v\,\pd_t\zeta+A^{\mathsf T}\nabla v\cdot\nabla\zeta}
  \ =\ \iint_\om F\cdot\nabla\zeta
  \qquad\text{for every }\zeta\in C^\infty_c(\om) .
\end{equation}
Such a $v$ exists and is unique. We briefly specify the reduction, since the
coefficient of the time derivative is not constant after flattening. Write
$\Pi(X,t)=(\Theta(X,t),t)$ in \eqref{E.LM}, put
$$\mathbb J:=D_X\Theta,\qquad
  \widetilde A:=\mathsf J\,\mathbb J^{-1}(A\circ\Pi)\mathbb J^{-\mathsf T},
  \qquad
  \mathbf b_0:=\pd_t(P_{\gamma x_n}\phi)\,e_n,$$
where $\mathsf J=\det\mathbb J$ is the quantity in \eqref{E.gamma}. The
pull-back $\widetilde v:=v\circ\Pi$ satisfies
\begin{equation}\label{E.PDpull}
 \mathsf J\pd_t\widetilde v
 +\divg_X(\widetilde A^{\mathsf T}\nabla_X\widetilde v)
 -\mathbf b_0\cdot\nabla_X\widetilde v
 =\divg_X\widetilde F
 \quad\hbox{in }\Uc,
\end{equation}
where $\widetilde F$ is the Piola transform of $F$ and
$\|\widetilde F\|_{L^2}\lesssim\|F\|_{L^2}$. The two geometrical coefficients
satisfy
\begin{equation}\label{E.PDstructure}
 \divg_X\mathbf b_0=\pd_t\mathsf J,\qquad
 x_n\br{|\nabla_X\mathsf J|+|\mathbf b_0|}
 +x_n^2|\pd_t\mathsf J|\le C(n,\ell).
\end{equation}
Moreover $\mathsf J\approx1$ and $\widetilde A$ is bounded and uniformly
elliptic. The constant-time-coefficient construction in
\cite[\S1.5]{AEN} uses hidden coercivity. Here we use a Galerkin construction
with the time-dependent mass matrix. On a finite interval $(t_-,T)$,
where $T$ is later than the time support of $\widetilde F$, choose nested
finite-dimensional spaces spanned by linearly independent functions
$w_j\in C_c^\infty(\R^n_+)$ whose span is dense in $H^1_0(\R^n_+)$.
Impose zero terminal data at $T$ and extend the approximants by zero
for $t\ge T$.
The mass matrix has entries $\int_{\R^n_+}\mathsf J w_jw_k$ and is uniformly
positive definite on each such space. Its entries are locally absolutely
continuous in time, since the basis functions are supported away from
$x_n=0$. Thus the finite-dimensional system has a solution.

Testing that system by its solution $v_m$ gives the backward energy identity
\begin{equation}\label{E.PDenergy}
 -\frac12\frac{\d}{\d t}\int_{\R^n_+}\mathsf J |v_m|^2
 +\int_{\R^n_+}\widetilde A^{\mathsf T}\nabla_Xv_m\cdot\nabla_Xv_m
 =\int_{\R^n_+}\widetilde F\cdot\nabla_Xv_m.
\end{equation}
Indeed, the two additional terms are
$\frac12\int_{\R^n_+}(\pd_t\mathsf J)|v_m|^2$ and
$-\frac12\int_{\R^n_+}(\divg_X\mathbf b_0)|v_m|^2$, which cancel by
\eqref{E.PDstructure}. Integrating from $t$ to $T$ and using ellipticity
gives bounds uniform in the dimension and in $t_-$. Weak compactness, first as
the Galerkin dimension tends to infinity and then as $t_-\to-\infty$, gives a solution
of \eqref{E.PDpull} with
$$\|\nabla_X\widetilde v\|_{L^2(\Uc)}
 +\sup_s\|\widetilde v(\cdot,s)\|_{L^2(\R^n_+)}
 \ \lesssim\ \|\widetilde F\|_{L^2(\Uc)}.$$
Write $\dot H^1_0(\R^n_+)$ for the completion of $C_c^\infty(\R^n_+)$
in the gradient norm and $\dot H^{-1}(\R^n_+)$ for its dual.
The singular drift passes to the weak formulation by Hardy's inequality
$\|\varphi/x_n\|_2\le2\|\nabla_X\varphi\|_2$ for
$\varphi\in\dot H^1_0(\R^n_+)$.

The equation gives
\[
 \mathsf J\pd_t\widetilde v
 =\divg_X(\widetilde F-\widetilde A^{\mathsf T}\nabla_X\widetilde v)
   +\mathbf b_0\cdot\nabla_X\widetilde v.
\]
Its right-hand side belongs to $L^2(\R;\dot H^{-1})$. For the drift term,
we use
\[
 \left|\int_{\R^n_+}(\mathbf b_0\cdot\nabla_X\widetilde v)\varphi\right|
 \lesssim\|\nabla_X\widetilde v\|_2\|\varphi/x_n\|_2
 \lesssim\|\nabla_X\widetilde v\|_2\|\nabla_X\varphi\|_2.
\]
Moreover \eqref{E.PDstructure} and Hardy give
\[
 \|\nabla_X(\varphi/\mathsf J)\|_2
 \lesssim\|\nabla_X\varphi\|_2+\|\varphi/x_n\|_2
 \lesssim\|\nabla_X\varphi\|_2.
\]
By duality we may divide the equation by $\mathsf J$, obtaining
$\pd_t\widetilde v\in L^2(\R;\dot H^{-1}(\R^n_+))$.
Time regularisation now justifies the same energy identity for weak
solutions. The terms containing $\pd_t\mathsf J$ are integrable by
$|\pd_t\mathsf J|\lesssim x_n^{-2}$ and Hardy. Applied to a difference
with zero terminal data, the identity proves uniqueness. On every interval
after the time support of $\widetilde F$, the Galerkin solution is zero, so
the limit has the same property.
Changing variables gives the solution specified above.\medskip

\noindent We say that $(PD_{p'})^{\LL^*}$ holds if for every
$F\in L^\infty_c(\om;\R^n)$ that solution satisfies, with
$\wt{\mathcal C}$ the Carleson functional \eqref{E.carlfun},
$$\big\|\Nt(v)\big\|_{L^{p'}(\pom)}\ \le\
  C\big\|\wt{\mathcal C}(\delta F)\big\|_{L^{p'}(\pom)} .$$

\begin{lemma}\label{L.PD}
$(D_{\LL^*})_{p'}$ implies $(PD_{p'})^{\LL^*}$.
\end{lemma}

\begin{proof}
This is the Whitney-averaged conclusion in the proof of
\cite[Lemma 3.8]{DLP2}, which follows the argument of \cite[\S3]{U}. In the
local part of that proof the average in $\Nt(v)$ permits the energy estimate
with the $L^2$ Whitney average of $F$. The remaining argument uses the Green
function estimates, boundary comparison and the doubling properties of
parabolic measure. These results hold on a $\Lip(1,\frac12)$ graph domain, so
the same proof applies on $\om$.
\end{proof}

\noindent We also use that for each fixed $C>0$ there are constants
$C_1,C'>0$, depending only on $C$, the graph character and the fixed aperture,
such that for $w\in L^2_{\loc}(\om)$, every $(P,\tau)\in\pom$ and every $r>0$,
\begin{equation}\label{E.conebd}
  \iint_{B_{Cr}(P,\tau)\cap\om}\abs w\,\d X\d t
  \ \le\ C'\,r\,\big\|\Nt(w)\big\|_{L^1(\Delta_{C_1r}(P,\tau))} .
\end{equation}
To prove this, take a sufficiently fine parabolic Whitney decomposition of
$\om$, and keep the boxes $W$ meeting $B_{Cr}(P,\tau)$. If $Z_W$ is the
centre of $W$ and $d_W=\delta(Z_W)$, we may arrange that
$W\subset B_{d_W/2}(Z_W)$ and $|W|\approx d_W^{n+2}$. Choose a nearest
boundary point $P_W$ to $Z_W$, and put $\Delta_W=\Delta_{ad_W/2}(P_W)$,
where $a$ is the original aperture in \eqref{E.cone}. The triangle inequality
gives $Z_W\in\Gam_a(q,\tau)$ for every $(q,\tau)\in\Delta_W$. Hence
\[
 \int_W|w|\lesssim d_W^{n+2}
       \inf_{\Delta_W}\Nt_a(w)
 \lesssim d_W\int_{\Delta_W}\Nt_a(w)\,\d\sigma.
\]
Here $d_W\lesssim r$ and $\Delta_W\subset\Delta_{C_1r}(P,\tau)$ for a
fixed $C_1$. At each dyadic scale the balls $\Delta_W$ have bounded overlap:
centres of the corresponding disjoint Whitney boxes lie within a distance
comparable to that scale of any point in their shadows. Summing over the
boxes at each scale and then using
$\sum_{2^{-k}r\lesssim r}2^{-k}r\lesssim r$ proves \eqref{E.conebd} with
the original aperture.\medskip

\subsection{Proof of Theorem \ref{T.locest}}\label{ss.locproof}

\begin{proof}[Proof of Theorem \ref{T.locest}]
Fix $K_0=K_0(n,\ell)$ large enough for \eqref{E.supu} below, and at least
$4$. We follow the scheme of \cite[Lemma 4.12]{DLP2}, in particular equations
(4.13) and (4.14) there.\medskip

Let $\varphi$ be a one-sided piecewise smooth cutoff on $\{t\le\tau\}$ with
$\varphi\equiv1$ on $J_{3r/2}$, $\varphi\equiv0$ outside
$J_{2r}\cap\{t\le\tau\}$, $\abs{\nabla\varphi}\lesssim r^{-1}$ and
$\abs{\pd_t\varphi}\lesssim r^{-2}$. We choose its transition regions with
positive separation from both $J_{5r/4}$ and the complement of $J_{4r}$.

Put
$$E_m:=\set{(X,t)\in J_r\cap\om:\ \delta(X,t)>m^{-1},\
 |\nabla u(X,t)|<m},\qquad
 F_m:=\nabla u\,\mathbf1_{E_m}.$$
By monotone convergence it is enough to obtain a bound independent of $m$ for
$\|\Nt(F_m)\|_{L^p}$. The vector-valued form of
\eqref{E.carlduality}, followed by truncation of the dual function to $E_m$,
gives $G=G_m\in L^\infty_c(\om;\R^n)$ such that
$$\|\wt{\mathcal C}(\delta G)\|_{L^{p'}}\le1,
  \qquad
  \|\Nt(F_m)\|_{L^p}
  \ \lesssim\ \iint_\om\nabla u\cdot G\,\varphi+m^{-1}.$$
We have changed the sign of $G$ if necessary. Its support is contained in
$E_m\subset J_r$, and hence
$$\iint_\om\nabla u\cdot G\,\varphi
  =\int_{-\infty}^\tau\!\!\int_{\om_t}\nabla(u\varphi)\cdot G .$$

Let $v$ be the finite-energy solution of $\LL^*v=-\divg G$ in $\om$ with
$v=0$ on $\pom$. We use $u\varphi$ in the weak equation \eqref{E.PDweak} for
$v$ and
$\varphi v$ in the weak equation for $u$. This cross-testing is justified by
pulling the two equations to $\Uc$, taking  averages in $t$, and passing
to the limit. Both test functions have zero lateral trace. Moreover $v=0$ for
$t>\tau$, and its terminal $L^2$ trace at $t=\tau$ is zero. Hence there is no
term from the upper time face. The two weak equations give the identity
\cite[(4.16)]{DLP2}, with the constant $c_u$ there equal to zero:
$$\iint\nabla u\cdot G\,\varphi
  \ =\ -\iint u\,v\,\pd_t\varphi
  \ +\ \iint A\nabla u\cdot\nabla\varphi\,v
  \ -\ \iint A\nabla\varphi\cdot\nabla v\,u
  \ =:\ I_0+I_1+I_2 .$$
Let $\mathcal A$ be the support of
$|\nabla\varphi|+|\pd_t\varphi|$. We choose two enlarged annular regions
$$\mathcal A\subset\mathcal A_1\subset\mathcal A'
 :=J_{4r}\cap\om\setminus J_{5r/4}$$
with successive separations comparable to $r$. At the upper time face we use
one-sided cylinders, or equivalently extend $v$ by zero to $t>\tau$.

\noindent\emph{The average of $u$.} Since $u$ vanishes on
$J_{K_0r}\cap\pom$, the
local boundedness estimate up to that boundary portion \cite[Ch.~III]{LSU} and
the Poincar\'e inequality give
\begin{equation}\label{E.supu}
  \sup_{J_{2r}\cap\om}\abs u\ \le\ C\fiint_{J_{4r}\cap\om}\abs u
  \ \le\ Cr\fiint_{J_{K_0r}\cap\om}\abs{\nabla u} .
\end{equation}
Here and below $C$ may depend on $K_0$. Over the time interval of $J_{4r}$ the
graph moves by at most $C\ell r$. Our choice of $K_0=K_0(n,\ell)$ therefore
places a quantitative part of the zero side of every relevant spatial slice
inside the spatial ball of $J_{K_0r}$. We apply the spatial Poincar\'e
inequality, slice by slice, to the extension of $u$ by zero across the graph.
This extension lies in $W^{1,2}$ because the trace of $u$ vanishes there.

\noindent\emph{The averages of $v$ and of $\nabla v$.} The source $G$ is
supported in $J_r$, so $v$ solves the
homogeneous adjoint equation on $\mathcal A'$ and vanishes on the boundary
portion there. Cover $\mathcal A_1$ by cubes of size $cr$ whose doubles stay
in $\mathcal A'$. The local boundedness estimate \cite[Ch.~III]{LSU}, in its
interior form on the cubes that do not meet $\pom$ and in its boundary form on
those that do, then \eqref{E.conebd} and $(PD_{p'})^{\LL^*}$, which holds by
Lemma \ref{L.PD}, give
\begin{equation}\label{E.supv}
  \sup_{\mathcal A_1}\abs v\ \le\ Cr^{-n-2}\iint_{\mathcal A'}\abs v
  \ \le\ Cr^{-n-1}\big\|\Nt(v)\big\|_{L^1(\Delta_{Cr})}
  \ \le\ Cr^{-(n+1)/p'} ,
\end{equation}
the last step by H\"older's inequality and
$\|\Nt(v)\|_{L^{p'}}\le C\|\wt{\mathcal C}(\delta G)\|_{L^{p'}}\le C$. The
Caccioppoli inequality \cite{Ar}, now on cubes covering $\mathcal A$ whose
doubles stay in $\mathcal A_1$, turns \eqref{E.supv} into
\begin{equation}\label{E.gradv}
  \iint_{\mathcal A}\abs{\nabla v}\ \le\ Cr^{n+2}
  \Big(\fiint_{\mathcal A}\abs{\nabla v}^2\Big)^{1/2}
  \ \le\ Cr^{n+1}\sup_{\mathcal A_1}\abs v
  \ \le\ Cr^{n+1}r^{-(n+1)/p'} ,
\end{equation}
which is $Cr^{(n+1)/p}$.\medskip

\noindent With this in hand the three terms are estimated, using
$\abs{\pd_t\varphi}\lesssim r^{-2}$ and $\abs{\nabla\varphi}\lesssim r^{-1}$
throughout. The term $I_0$ is bounded by \eqref{E.supu}, \eqref{E.supv} and
$\abs{\mathcal A}\le Cr^{n+2}$. The term $I_1$ is bounded by \eqref{E.supv} and
$\abs{\mathcal A}\le Cr^{n+2}$; only the integral of $|\nabla u|$ over
$\mathcal A$ is used. The term $I_2$ is bounded by
\eqref{E.supu} and \eqref{E.gradv}. Each of the three is at most
$$Cr^{(n+1)/p}\fiint_{J_{K_0r}\cap\om}\abs{\nabla u}.$$
The constant is independent of $m$. Letting $m\to\infty$ in the first part of
the proof and using monotone convergence gives \eqref{E.locest}.
\end{proof}

\section{Sobolev extrapolation}\label{S.extrapolation}

We prove the part of Theorem \ref{T.dual} in which the exponent at which the
Regularity problem is already known is smaller than $p$. The real-variable
argument is a Sobolev version of \cite[Theorem 3.1]{Shen}. We state it with an
arbitrary base exponent. This is needed because the known Regularity estimate
need not be an $L^2$ estimate.

We use parabolic cubes in $\R^{n-1}\times\R$. Thus a cube of radius $r$ has
spatial side length comparable to $r$, time length comparable to $r^2$, and
measure comparable to $r^{n+1}$. If $Q$ is such a cube, $aQ$ denotes its
concentric parabolic dilate.

\begin{lemma}[localisation of the datum]\label{L.datumloc}
Let $1<s<\infty$. For every $f\in C_c^\infty(\R^{n-1}\times\R)$ there is a
nonnegative function $G_f\in L^s$ such that
\begin{equation}\label{E.Gfnorm}
 \|G_f\|_{L^s}\ \le\ C_s\big(\|\nabla_yf\|_{L^s}
                  +\|\dhalf f\|_{L^s}\big).
\end{equation}
For every parabolic cube $Q$ and every fixed $a>1$ there are a constant $c_Q$
and a smooth cutoff $\eta_Q$, equal to one on $aQ$ and supported in $A_0Q$,
where $A_0=A_0(a)$, such that, for every $1<p_0\le s$,
\begin{equation}\label{E.datumloc}
 \big\|\eta_Q(f-c_Q)\big\|_{\dot L^{p_0}_{1,1/2}}
 \ \le\ C\left(\int_{A_0 Q}G_f^{p_0}\right)^{1/p_0}.
\end{equation}
The constants are independent of $f$ and $Q$.
\end{lemma}

\begin{proof}
Let $\mathcal M_y$ be the uncentred Hardy--Littlewood maximal operator
over spatial cubes, at fixed time, and put
\[
 G_f:=\mathcal M_y(|\nabla_y f|)+\mathcal M_y(\mathcal S_t f).
\]
Lemma \ref{L.timesquare}, the spatial maximal theorem and Fubini give
\eqref{E.Gfnorm} for every $s>1$. Also
$G_f\ge|\nabla_y f|+\mathcal S_t f$ almost everywhere.

Fix a cube $Q$ with centre $(y_Q,t_Q)$ and radius $r$. We use the
normalisation in which its spatial half-side length is $r$ and its time
half-length is $r^2$. Put $R=4ar$ and write
\[
 Q^*:=4aQ=B\times I,\qquad
 B=\{y:|y-y_Q|_\infty<R\},\qquad
 I=(t_Q-R^2,t_Q+R^2).
\]
Set $c_Q=f_{Q^*}$ and $f_B(t)=\fiint_B f(y,t)\,\d y$.
Spatial Poincar\'e and telescoping give
\[
 |f(y,t)-f_B(t)|\lesssim R\mathcal M_y(|\nabla_y f|)(y,t),
 \qquad y\in B.
\]
Indeed, telescope the averages on cubes centred at $y$ with radii
$2R,R,R/2,\ldots$. Poincar\'e bounds the sum by the right-hand side.
The first cube contains $B$ with comparable volume, so Poincar\'e also
bounds the difference of their averages by the same quantity.

For $x\in B$ and $t\in I$, Cauchy--Schwarz gives
\begin{align*}
 \left|f(x,t)-\fiint_I f(x,\tau)\,\d\tau\right|
 &\le \frac1{|I|}
 \left(\int_I\frac{|f(x,t)-f(x,\tau)|^2}{|t-\tau|^2}\,\d\tau\right)^{1/2}
 \left(\int_I|t-\tau|^2\,\d\tau\right)^{1/2}\\
 &\lesssim R\mathcal S_t f(x,t).
\end{align*}
Average in $x\in B$ and use $y\in B$. Combining the two estimates yields
\begin{equation}\label{E.datumosc}
 |f(y,t)-c_Q|\lesssim R G_f(y,t)
 \qquad\hbox{a.e. on }Q^*.
\end{equation}

Choose a smooth product cutoff $0\le\eta_Q\le1$, equal to one on $aQ$
and compactly supported in $2aQ$, with
$|\nabla_y\eta_Q|\lesssim R^{-1}$ and
$|\pd_t\eta_Q|\lesssim R^{-2}$. Write $b=\eta_Q(f-c_Q)$.
The spatial product rule and \eqref{E.datumosc} give
\begin{equation}\label{E.datumspatial}
 \|\nabla_y b\|_{L^{p_0}}
 \lesssim\left(\int_{Q^*}G_f^{p_0}\right)^{1/p_0}.
\end{equation}

We estimate the time derivative through Lemma \ref{L.timesquare}.
The cutoff bounds imply
\[
 \mathcal S_t\eta_Q(y,t)^2
 \lesssim\int_\R\frac{\min\{1,|h|^2/R^4\}}{|h|^2}\,\d h
 \lesssim R^{-2}.
\]
For all $\tau\in\R$ we have the identity
\begin{align*}
 b(y,t)-b(y,\tau)
 &=\eta_Q(y,\tau)\big(f(y,t)-f(y,\tau)\big)\\
 &\quad+\big(\eta_Q(y,t)-\eta_Q(y,\tau)\big)\big(f(y,t)-c_Q\big).
\end{align*}
Minkowski's inequality in \eqref{E.timesquaredef} and \eqref{E.datumosc}
therefore give
\begin{equation}\label{E.datumlocaltime}
 \mathcal S_t b(y,t)
 \le\mathcal S_t f(y,t)+|f(y,t)-c_Q|\mathcal S_t\eta_Q(y,t)
 \lesssim G_f(y,t),\qquad (y,t)\in Q^*.
\end{equation}

It remains to estimate the time square function outside $I$. Fix $y\in B$
and put $L=|I|=2R^2$. For every $v\in I$, the triangle inequality in
$L^2(I)$ gives
\[
 \|b(y,\cdot)\|_{L^2(I)}
 \le \|f(y,\cdot)-c_Q\|_{L^2(I)}
 \le L^{1/2}|f(y,v)-c_Q|+L\mathcal S_t f(y,v).
\]
In the last step we used $|\tau-v|\le L$ for $\tau,v\in I$.
Average in $v$, apply \eqref{E.datumosc}, and then use H\"older to obtain
\begin{equation}\label{E.datumtimeLtwo}
 \|b(y,\cdot)\|_{L^2(I)}
 \lesssim L\fiint_I G_f(y,v)\,\d v
 \le L^{1-1/p_0}
       \left(\int_I G_f(y,v)^{p_0}\,\d v\right)^{1/p_0}.
\end{equation}
This estimate holds also when $1<p_0<2$.

For $j\ge0$, let
\[
 I_j=\{t:2^jR^2\le|t-t_Q|<2^{j+1}R^2\}.
\]
These annuli cover $\R\setminus I$ up to endpoints. The time support of
$b$ is contained in $[t_Q-R^2/4,t_Q+R^2/4]$. Hence, for $t\in I_j$,
$b(y,t)=0$ and
\[
 \mathcal S_t b(y,t)
 \lesssim (2^jR^2)^{-1}\|b(y,\cdot)\|_{L^2(I)}.
\]
Since $|I_j|\approx2^jR^2$, \eqref{E.datumtimeLtwo} yields
\[
 \int_{I_j}\mathcal S_t b(y,t)^{p_0}\,\d t
 \lesssim 2^{j(1-p_0)}\int_I G_f(y,v)^{p_0}\,\d v.
\]
Sum over $j\ge0$, using $p_0>1$, and integrate in $y\in B$.
Outside $B$ the function $b(y,\cdot)$ vanishes identically.
Together with \eqref{E.datumlocaltime}, this proves
\[
 \|\mathcal S_t b\|_{L^{p_0}}
 \lesssim\left(\int_{Q^*}G_f^{p_0}\right)^{1/p_0}.
\]
Lemma \ref{L.timesquare}, \eqref{E.datumspatial} and \eqref{E.norms}
now give \eqref{E.datumloc} with $A_0=4a$. The function $G_f$ is
independent of $Q$, $a$ and $p_0$, and the estimate holds for every
$1<p_0\le s$.
\end{proof}

\begin{theorem}[Sobolev extrapolation]\label{T.SobolevShen}
Let $1<p_0<p_1\le\infty$, and let $\mathcal T$ be a nonnegative-valued
operator defined on $C_c^\infty(\R^{n-1}\times\R)+\R$. Suppose
that
\begin{equation}\label{E.Tsublinear}
 |\mathcal Tf-\mathcal Tg|\le\mathcal T(f-g),\qquad
 \mathcal T(f+c)=\mathcal Tf,
\end{equation}
and that
\begin{equation}\label{E.Tpzero}
 \|\mathcal Tf\|_{L^{p_0}}
 \ \le\ C_0\|f\|_{\dot L^{p_0}_{1,1/2}}.
\end{equation}
Assume that there are $a_2>a_1>1$ and $C_1$ such that, whenever $f$ is
constant on $a_2Q$,
\begin{equation}\label{E.Toffdiag}
 \left(\fiint_Q|\mathcal Tf|^{p_1}\right)^{1/p_1}
 \ \le\ C_1
 \left(\fiint_{a_1Q}|\mathcal Tf|^{p_0}\right)^{1/p_0}
\end{equation}
for every parabolic cube $Q$, with the usual supremum interpretation when
$p_1=\infty$. Then, for every $p_0<s<p_1$,
\begin{equation}\label{E.Ts}
 \|\mathcal Tf\|_{L^s}
 \ \le\ C_s\|f\|_{\dot L^s_{1,1/2}}.
\end{equation}
\end{theorem}

\noindent The hypothesis \eqref{E.Toffdiag} carries no term in $f$ itself,
where \cite[(3.2)]{Shen} carries
\[
 \sup_{Q'\supset Q}\big(\fiint_{Q'}|f|^{p_0}\big)^{1/p_0}.
\]
Ours is therefore a stronger hypothesis, and it is what Lemma
\ref{L.offdiag} gives us. Here $G_f$ plays the role of the omitted term from
\cite{Shen}.

\begin{proof}
We give the modification of Shen's good-$\lambda$ proof. Fix
$p_0<s<p_1$, take $G_f$ from Lemma \ref{L.datumloc}, and put
\[
 E(\lambda)=\{\mathcal M(|\mathcal Tf|^{p_0})>\lambda\},\qquad
 H=\mathcal M(G_f^{p_0}),
\]
where $\mathcal M$ is the uncentred parabolic Hardy--Littlewood maximal
operator over cubes.
By \eqref{E.Tpzero} we have $\mathcal Tf\in L^{p_0}$, so
$\mathcal M(|\mathcal Tf|^{p_0})$ lies in weak $L^1$ and $E(\lambda)$ has
finite measure for every $\lambda>0$. Hence
\[
 \int_0^N\lambda^{s/p_0-1}\abs{E(\lambda)}\d\lambda\ <\ \infty
 \qquad\hbox{for every }N<\infty ,
\]
the integral converging at the origin because $s>p_0$. All
distribution-function integrals below are truncated at such a height $N$, the
estimates are uniform in $N$, and we let $N\to\infty$ by monotone convergence
at the end. The argument therefore does not assume the $L^s$ membership that
it is proving.

Put $W=|\mathcal Tf|^{p_0}$. Decompose the open set $E(\lambda)$ into
maximal parabolic dyadic cubes $Q$ for which
$\overline{8Q}\subset E(\lambda)$. These cubes are disjoint up to their
boundaries and cover $E(\lambda)$ up to a null set. If $\widehat Q$ is
the dyadic parent of $Q$, maximality gives a point
\[
 z_Q\in\overline{8\widehat Q}\setminus E(\lambda)
 \subset20Q,\qquad \mathcal M W(z_Q)\le\lambda.
\]
Here each parent has twice the spatial side lengths and four times
the time length of its children.

For $z\in Q$, every cube $P$ containing $z$ with radius at least $r(Q)$
can be enlarged by a fixed factor to contain $z_Q$. Its average of $W$
is therefore at most $C_*\lambda$, with a fixed $C_*$. Every smaller
cube $P$ containing $z$ lies in $3Q$. Consequently, if $B>C_*$,
\begin{equation}\label{E.maxlocalised}
 Q\cap E(B\lambda)
 \subset\{z\in Q:\mathcal M(W\mathbf1_{3Q})(z)>B\lambda\}.
\end{equation}
Also, for every fixed $c\ge1$, enclosing $cQ$ and $z_Q$ in a common
cube of comparable size gives
\begin{equation}\label{E.whitneyaverage}
 \fiint_{cQ}W\le C_c\lambda.
\end{equation}

Suppose $Q\cap\{H\le\gamma\lambda\}\ne\varnothing$, where $0<\gamma\le1$.
Apply Lemma \ref{L.datumloc} with $a=3a_2$ and write
\[
 b_Q=\eta_Q(f-c_Q),\qquad h_Q=f-b_Q.
\]
Thus $h_Q$ is constant on $3a_2Q$. The enclosing cube $A_0Q$ contains
a point where $H\le\gamma\lambda$, so
$\int_{A_0Q}G_f^{p_0}\le A_0^{n+1}\gamma\lambda|Q|$. Hence
\begin{equation}\label{E.bQ}
 \|b_Q\|_{\dot L^{p_0}_{1,1/2}}^{p_0}
 \ \le\ C\gamma\lambda|Q|.
\end{equation}
By \eqref{E.whitneyaverage} with $c=3a_1$,
\eqref{E.Tsublinear}, \eqref{E.Tpzero} and \eqref{E.bQ},
\[
 \fiint_{3a_1Q}|\mathcal T h_Q|^{p_0}
 \lesssim\fiint_{3a_1Q}W
       +|Q|^{-1}\|\mathcal T b_Q\|_{L^{p_0}}^{p_0}
 \lesssim\lambda.
\]
We may apply \eqref{E.Toffdiag} on $3Q$, since its required constant-data
patch is precisely $3a_2Q$. For $p_1<\infty$ this gives
\[
 \int_{3Q}|\mathcal T h_Q|^{p_1}
 \lesssim\lambda^{p_1/p_0}|Q|.
\]
In \eqref{E.maxlocalised}, use
$W\le2^{p_0-1}(|\mathcal T b_Q|^{p_0}
                  +|\mathcal T h_Q|^{p_0})$.
The weak $(1,1)$ bound for $\mathcal M$ applied to
$|\mathcal T b_Q|^{p_0}$, and the weak
$(p_1/p_0,p_1/p_0)$ bound applied to
$|\mathcal T h_Q|^{p_0}\mathbf1_{3Q}$, yield
\begin{equation}\label{E.localgoodlambda}
 |Q\cap E(B\lambda)|
 \ \le\ C\left(\frac\gamma B+B^{-p_1/p_0}\right)|Q|.
\end{equation}
If $p_1=\infty$, the off-diagonal bound instead gives
$\|\mathcal T h_Q\|_{L^\infty(3Q)}^{p_0}\lesssim\lambda$.
Its contribution to the maximal function cannot exceed $C\lambda$,
so the second term in \eqref{E.localgoodlambda} is absent for large $B$.

Choose $B=(2\varepsilon)^{-p_0/s}$. Since $p_1>s$, we may first choose
$\varepsilon>0$ and then $\gamma>0$ so that the right side of
\eqref{E.localgoodlambda} is at most $\varepsilon|Q|$. Summing over the
maximal cubes yields
\[
 |E(B\lambda)|\le\varepsilon|E(\lambda)|
                 +|\{H>\gamma\lambda\}|.
\]
Put $\alpha=s/p_0$ and
$J(N)=\int_0^N\lambda^{\alpha-1}|E(\lambda)|\,\d\lambda$.
Multiplying the last inequality by $\lambda^{\alpha-1}$ and integrating
from $0$ to $N$ gives
\[
 B^{-\alpha}J(BN)
 \le\varepsilon J(N)+\frac{\gamma^{-\alpha}}{\alpha}\|H\|_{L^\alpha}^{\alpha}.
\]
Since $J(N)\le J(BN)<\infty$ and $\varepsilon B^\alpha=1/2$,
we may absorb the first term. Before letting
$N\to\infty$ the resulting estimate is
\[
 \|\min\{\mathcal M(|\mathcal Tf|^{p_0}),N\}\|_{L^{s/p_0}}^{s/p_0}
 \lesssim \|\mathcal M(G_f^{p_0})\|_{L^{s/p_0}}^{s/p_0}.
\]
The strong $L^{s/p_0}$ bound for $\mathcal M$, \eqref{E.Gfnorm}, and monotone
convergence now give
\[
 \|\mathcal Tf\|_{L^s}^s
 \le\|\mathcal M(|\mathcal Tf|^{p_0})\|_{L^{s/p_0}}^{s/p_0}
 \lesssim\|G_f\|_{L^s}^s
 \lesssim\|f\|_{\dot L^s_{1,1/2}}^s.
\]
The argument for $p_1=\infty$ is the same, with the second term in
\eqref{E.localgoodlambda} absent once $B$ is large enough.
\end{proof}

We next put Theorem \ref{T.locest} into the form required in
\eqref{E.Toffdiag}. For this purpose we use rectangular cones in the fixed
graph coordinates. Write
\[
 Q_R(z):=\set{(y,s):|y_j-z_j|<R\ (1\le j\le n-1),\ |s-z_t|<R^2},
 \qquad z=(z_1,\dots,z_{n-1},z_t).
\]
Put $a_\square=2+a$, where $a$ is the aperture in \eqref{E.cone}. For
$h\in L^2_{\loc}(\om)$ define
\begin{equation}\label{E.graphmax}
 \mathcal N(h)(\pi(z)):=
 \sup_{\substack{(X,t)=(x,x_n,t)\in\om\\z\in Q_{a_\square\delta(X,t)}(x,t)}}
 \left(\fiint_{B_{\delta(X,t)/2}(X,t)}|h|^2\right)^{1/2}.
\end{equation}
The distance to the boundary is still the original parabolic distance, and
the Whitney averages have not changed. The graph bound gives
$x_n-\phi(x,t)\lesssim\delta(X,t)$, and consequently, for a fixed $a'>a$,
\begin{equation}\label{E.graphmaxcompare}
 \Nt_a(h)\le\mathcal N(h)\le\Nt_{a'}(h),\qquad
 \|\mathcal N(h)\|_{L^v(\pom)}\approx\|\Nt_a(h)\|_{L^v(\pom)},
 \quad 1<v<\infty.
\end{equation}
The norm comparison uses only the global change-of-aperture estimate. Thus
this auxiliary maximal function gives the same solvability estimates as
\eqref{E.Ntilde}. Its boundary shadow at $(X,t)$ is precisely
$\pi(Q_{a_\square\delta(X,t)}(x,t))$. This is the graph-coordinate version of the
enlarged-shadow construction in \cite[\S2.1]{DN}.

The remaining ingredient is the cone-volume estimate \eqref{E.conebd}.
It is here that the $L^1$ average on the right of
\eqref{E.locest} is used. With the $L^2$ average one would have to pass from
it to an $L^1$ one, which is the boundary weak reverse H\"older estimate for
gradients of \cite[Lemma 3.6]{DN}, proved there on a Lipschitz cylinder but not for our class of boundaries.

\begin{lemma}[off-diagonal reverse H\"older estimate]\label{L.offdiag}
Assume $(D_{\LL^*})_{p'}$ for some $1<p<\infty$. There are $\bar p>p$ and
$a_2>a_1>1$, none of them depending on $p_0$, such that for every $p_0>1$ the
following holds. If $w$ solves $\LL w=0$ and has constant boundary data on
$\pi(a_2Q)$, then
\begin{equation}\label{E.Nreverse}
 \left(\fiint_{\pi(Q)}\mathcal N(\nabla w)^{\bar p}\,\d\sigma\right)^{1/\bar p}
 \ \le\ C
 \left(\fiint_{\pi(a_1Q)}\mathcal N(\nabla w)^{p_0}\,\d\sigma\right)^{1/p_0}.
\end{equation}
Here $Q$ and its dilates are cubes in the fixed graph coordinates, and
$\mathcal N$ is the auxiliary maximal function \eqref{E.graphmax} in those
coordinates. The constants $a_1,a_2$ depend only on the geometry and the
fixed apertures.
\end{lemma}

\begin{proof}
For each time-compatible pole $Z$, let
$k^{*,Z}:=\d\omega^{*,Z}/\d\sigma$ be the adjoint parabolic Poisson kernel.
We write $k^{*,Z}\in B_v(\d\sigma)$ when the reverse H\"older estimate
\[
 \left(\fiint_\Delta |k^{*,Z}|^v\,\d\sigma\right)^{1/v}
 \ \lesssim\ \fiint_\Delta k^{*,Z}\,\d\sigma
\]
holds on the relevant time-lagged surface balls $\Delta$.
Solvability of $(D_{\LL^*})_{p'}$ gives
$k^{*,Z}\in B_p(\d\sigma)$ uniformly on the corresponding time-lagged
surface balls, the adjoint parabolic measure being doubling there by
\cite[Theorem 3.2]{N}. Gehring's lemma is available on that space of
homogeneous type and raises the exponent. Hence
$k^{*,Z}\in B_{\bar p}(\d\sigma)$ uniformly and
$(D_{\LL^*})_{\bar p'}$ is solvable
for some $\bar p>p$. This is the deduction made in \cite[\S5]{DN}, and the
weight theory behind it is \cite[Theorem 6.1]{N}. We may therefore apply
Theorem \ref{T.locest} with the exponent $\bar p$. Everything below rests on
the strict inequality $\bar p>p$, since \eqref{E.Toffdiag} is of no use to
Theorem \ref{T.SobolevShen} unless $p_1$ exceeds $s$.

Subtract the constant boundary value from $w$, write $Q=Q_r(x_Q,t_Q)$,
and split $\mathcal N(\nabla w)$ into the suprema $\mathcal N_{<r}$ and
$\mathcal N_{\ge r}$ over centres with $\delta<r$ and $\delta\ge r$,
respectively. We may assume that the right-hand side of \eqref{E.Nreverse}
is finite.\medskip

\noindent\emph{The away part of $\mathcal N$.} Let $Z=(x,x_n,t)$ be a centre contributing
to $\mathcal N_{\ge r}$ at a point of $\pi(Q)$, and put
\[
 E_Z:=\pi\big(Q_{a_\square\delta(Z)}(x,t)\cap3Q\big),\qquad
 H(Z):=\left(\fiint_{B_{\delta(Z)/2}(Z)}|\nabla w|^2\right)^{1/2}.
\]
The cube $Q_{a_\square\delta(Z)}(x,t)$ meets $Q$ and has radius at least $r$.
Its intersection with $3Q$ has spatial side lengths at least $r$ and time
length at least $r^2$. Thus $\sigma(E_Z)\gtrsim\sigma(\pi(Q))$.
Moreover, $\mathcal N(\nabla w)\ge H(Z)$ everywhere on $E_Z$.
For any fixed $a_1\ge3$ this gives
\[
 \sup_{\pi(Q)}\mathcal N_{\ge r}
 \lesssim\fiint_{\pi(a_1Q)}\mathcal N(\nabla w)\,\d\sigma.
\]
The constant is uniform in $Z$, even when $\delta(Z)/r$ is arbitrarily
large.\medskip

\noindent\emph{The near part of $\mathcal N$.} Every Whitney ball contributing to
$\mathcal N_{<r}$ at a point of $\pi(Q)$ is contained in a single backward
boundary cylinder $J_R(P_Q,\tau_Q)$, where
\[
 \tau_Q=t_Q+(a_\square^2+2)r^2,\qquad
 P_Q=(x_Q,\phi(x_Q,\tau_Q)),\qquad R=C_0r.
\]
Indeed, its time coordinates lie between
$t_Q-(a_\square^2+2)r^2$ and $\tau_Q$, and its spatial coordinates are within
$Cr$ of $P_Q$ by the graph bound. Taking $C_0=C_0(n,\ell,a_\square)$ large enough
gives the inclusion. Choose $a_1\ge3$ large enough to contain, in
$\pi(a_1Q)$, the surface ball used when \eqref{E.conebd} is applied to
$J_{K_0R}\cap\om$. Then choose $a_2>a_1$ large enough that
$J_{K_0R}\cap\pom\subset\pi(a_2Q)$. These choices do not depend on $p_0$.

The full Whitney-ball inclusion gives
$\mathcal N_{<r}\le\mathcal N(\nabla w\,\mathbf1_{J_R})$ on $\pi(Q)$.
Using the global norm comparison \eqref{E.graphmaxcompare}, Theorem
\ref{T.locest}, and $R\approx r$, we obtain
\begin{align*}
 \left(\fiint_{\pi(Q)}\mathcal N_{<r}^{\bar p}\,\d\sigma\right)^{1/\bar p}
 &\lesssim r^{-(n+1)/\bar p}
      \|\Nt(\nabla w\,\mathbf1_{J_R})\|_{L^{\bar p}(\pom)}\\
 &\lesssim\fiint_{J_{K_0R}\cap\om}|\nabla w|\\
 &\lesssim\fiint_{\pi(a_1Q)}\mathcal N(\nabla w)\,\d\sigma.
\end{align*}
The last inequality is \eqref{E.conebd} and $\Nt\le\mathcal N$.
Combining the two parts and applying H\"older's inequality to this last
boundary average proves \eqref{E.Nreverse} for every $p_0>1$.
\end{proof}

\begin{proposition}\label{P.dualqp}
Under the hypotheses of Theorem \ref{T.dual}, suppose that $(R_{\LL})_q$ is
solvable for some $1<q<p$. Then $(R_{\LL})_p$ is solvable.
\end{proposition}

\begin{proof}
For $f\in C_c^\infty(\R^{n-1}\times\R)$ let $u_f$ be the parabolic-measure
solution with boundary datum $f\circ\pi^{-1}$, and set
\[
 \mathcal Tf:=\mathcal N(\nabla u_f)\circ\pi.
\]
Linearity of the solution operator gives \eqref{E.Tsublinear}, and constants
are annihilated. Surface measure in graph coordinates is comparable to
Lebesgue measure. Thus solvability of $(R_{\LL})_q$ and the global norm
comparison \eqref{E.graphmaxcompare} give \eqref{E.Tpzero} with $p_0=q$.

If $f$ is constant on $a_2Q$, then $u_f$ minus that constant has zero boundary
data on the corresponding boundary patch. Lemma \ref{L.offdiag} gives
\eqref{E.Toffdiag} with $p_0=q$ and $p_1=\bar p$. Since
$q<p<\bar p$, Theorem \ref{T.SobolevShen} gives
\[
 \|\Nt(\nabla u_f)\|_{L^p(\pom)}
 \le\|\mathcal N(\nabla u_f)\|_{L^p(\pom)}
 \lesssim\|f\|_{\dot L^p_{1,1/2}(\R^{n-1}\times\R)}.
\]
The reduction following Definition \ref{D.Rp} turns this into $(R_{\LL})_p$.
\end{proof}

We turn to the opposite order of the exponents (i.e., $p<q$). If $R>0$, we write $\Nt^R$
for the maximal function in \eqref{E.Ntilde} with the supremum restricted to
points for which $\delta(X,t)<R$.

\begin{lemma}[the Hardy--Sobolev endpoint]\label{L.atomendpoint}
Assume $(D_{\LL^*})_{p'}$ for some $1<p<\infty$ and $(R_{\LL})_q$ for some
$q>1$. Fix $\kappa>\max\{q,n+1\}$. Let $a$ be a smooth
$(1,\tfrac12,\kappa)$-atom associated with a parabolic ball $B=B_r(z_0)$,
and let $u_a$ be the parabolic-measure solution with datum $a\circ\pi^{-1}$.
Then
\begin{equation}\label{E.atomendpoint}
 \|\Nt(\nabla u_a)\|_{L^1(\pom)}\ \le\ C.
\end{equation}
The constant is independent of $a$, $B$ and $r$.
\end{lemma}

\begin{proof}
Put $E_B:=\pi(B)$. By \cite[Lemma 6.6]{DN}, a
$(1,\tfrac12,\kappa)$-atom is, up to a fixed constant, a
$(1,\tfrac12,q)$-atom. Consequently
\[
 \|a\|_{\dot L^q_{1,1/2}}\lesssim |B|^{-1/q'}.
\]
Solvability of $(R_{\LL})_q$ and H\"older's inequality therefore give
\begin{equation}\label{E.atomnear}
 \|\Nt(\nabla u_a)\|_{L^1(\pi(8B))}
 \ \lesssim\ |B|^{1/q'}\|a\|_{\dot L^q_{1,1/2}}\ \lesssim\ 1.
\end{equation}

Lemma \ref{L.atomsup} gives the pointwise Poincar\'e estimate
\eqref{E.atomsup}.

Write $a=a^+-a^-$ and let $u_+$ and $u_-$ be the nonnegative
parabolic-measure solutions with these data. Thus
$|u_a|\le u_++u_-$. For every $R\ge Cr$, choose a time-forward corkscrew
$Z_R$ at scale $R$ relative to the point $\pi(z_0)$. Positivity,
\eqref{E.atomsup}, and \eqref{E.omegadecay} give
\begin{equation}\label{E.atompole}
 u_+(Z_R)+u_-(Z_R)
 \ \le\ \|a\|_\infty\omega^{Z_R}(E_B)
 \ \lesssim\ r^\beta R^{-n-\beta}.
\end{equation}
Here we used that $\pi(B)$ is contained in a surface ball of radius $Cr$.
This pole estimate is the off-diagonal input.

Let $R_j=2^jr$ and
\[
 \mathcal A_j:=\pi(2^{j+1}B\setminus2^jB),\qquad j\ge3.
\]
Each $\mathcal A_j$ is covered by a bounded number of surface balls of radius
$cR_j$, with $c>0$ small. The cone portions above these balls with
$\delta<cR_j$ are then contained in a bounded number of backward boundary
cylinders of radius $cR_j$ whose $K_0$-fold dilates still do not meet
$E_B$, which fixes how small $c$ has to be. On each such cylinder $u_a$ has
zero boundary data. All cylinders in the next two displays have radii equal
to fixed multiples of this choice of $cR_j$. Theorem \ref{T.locest},
H\"older's inequality, Cauchy--Schwarz and boundary Caccioppoli give
\begin{align}
 \|\Nt^{cR_j}(\nabla u_a)\|_{L^1(\mathcal A_j)}
 &\lesssim R_j^{(n+1)/p'}R_j^{(n+1)/p}
       \fiint_{J_{CR_j}\cap\om}|\nabla u_a|\notag\\
 &\lesssim R_j^n\sup_{J_{C'R_j}\cap\om}|u_a|
 \ \lesssim\ r^\beta R_j^{-\beta}.
 \label{E.atomannlocal}
\end{align}
Here and below a finite sum over the covering cylinders is suppressed. To
justify the last inequality, apply \eqref{E.boundaryCarleson} separately to
$u_+$ and $u_-$, choosing its terminal corkscrew later in time than all the
enlarged cylinders. A fixed additional Harnack chain at scale $R_j$ gives
\[
 \sup_{J_{C'R_j}\cap\om}|u_a|
 \lesssim u_+(Z_{R_j})+u_-(Z_{R_j})
 \lesssim r^\beta R_j^{-n-\beta}
\]
by \eqref{E.atompole}.

For the remaining cone portion, interior Caccioppoli applied on the Whitney
ball gives the same estimate. Indeed, if its centre has
$d=\delta(X,t)\ge cR_j$, local boundedness and a time-forward Harnack chain at
scale $d$ give
\[
 \left(\fiint_{B_{d/2}(X,t)}|\nabla u_a|^2\right)^{1/2}
 \lesssim d^{-1}\big(u_+(Z_d)+u_-(Z_d)\big)
 \lesssim r^\beta d^{-n-\beta-1}.
\]
Taking the supremum over $d\ge cR_j$ gives
\[
 \Nt\big(\nabla u_a\mathbf1_{\{\delta\ge cR_j\}}\big)
 \ \lesssim\ r^\beta R_j^{-n-\beta-1}.
\]
Since $\sigma(\mathcal A_j)\lesssim R_j^{n+1}$, this part also has $L^1$ norm
at most $Cr^\beta R_j^{-\beta}$. Adding it to \eqref{E.atomannlocal} we have
proved
\begin{equation}\label{E.atomannulus}
 \|\Nt(\nabla u_a)\|_{L^1(\mathcal A_j)}\ \lesssim\ 2^{-j\beta}.
\end{equation}
Summing \eqref{E.atomannulus} over $j\ge3$ and adding \eqref{E.atomnear}
proves \eqref{E.atomendpoint}.
\end{proof}

\begin{proposition}\label{P.dualqabove}
Under the hypotheses of Theorem \ref{T.dual}, suppose that $(R_{\LL})_q$ is
solvable for some $q>p$. Then $(R_{\LL})_p$ is solvable.
\end{proposition}

\begin{proof}
Fix $\kappa>\max\{q,n+1\}$ and put
\[
 \mathcal Tf:=\Nt(\nabla u_f)\circ\pi.
\]
Solvability of $(R_{\LL})_q$ gives
\[
 \|\mathcal Tf\|_{L^q}\lesssim\|f\|_{\dot L^q_{1,1/2}}.
\]
Lemma \ref{L.atomendpoint} gives $\|\mathcal Ta\|_{L^1}\lesssim1$ for every
smooth $(1,\tfrac12,\kappa)$-atom.

The operator $\mathcal T$ is linearizable. Indeed, after choosing measurably
the Whitney region in the supremum defining $\Nt$, its $L^2$ average is the
supremum of the pairings with $L^2$ vector fields of norm one. Each such
pairing is linear in the solution and hence in the datum. Theorem
\ref{T.finiteinterp}, with $s=p$, now gives
\[
 \|\Nt(\nabla u_f)\|_{L^p(\pom)}
 \ \lesssim\ \|f\|_{\dot L^p_{1,1/2}(\R^{n-1}\times\R)}.
\]
The reduction following Definition \ref{D.Rp} turns this into
$(R_{\LL})_p$.
\end{proof}

Propositions \ref{P.dualqp} and \ref{P.dualqabove}, together with the trivial
case $q=p$, prove Theorem \ref{T.dual}.

\end{document}